\documentclass[12pt,a4paper]{article}
\usepackage{amsmath,amssymb,latexsym, amsfonts, amsthm,  mathrsfs,euscript,color,stmaryrd,mathabx}
\usepackage[top=2.5cm, bottom=2.5cm, left=2.5cm, right=2.5cm]{geometry}
\usepackage{pdfsync}

\numberwithin{equation}{section}

\theoremstyle{plain}
\newtheorem{theorem}[equation]{Theorem}
\newtheorem{proposition}[equation]{Proposition}
\newtheorem{lemma}[equation]{Lemma}
\newtheorem{corollary}[equation]{Corollary}
\newtheorem{conjecture}[equation]{Conjecture}
\newtheorem*{theorem*}{Theorem}
\theoremstyle{definition}
\newtheorem{notation}[equation]{Notation}

\newtheorem{definition}[equation]{Definition}

\newtheorem{remark}[equation]{Remark}

\newtheorem{example}[equation]{Example}

\theoremstyle{plain}

\theoremstyle{definition}

\newcommand{\be}{\begin{enumerate}}
\newcommand{\ee}{\end{enumerate}}
\newcommand{\bi}{\begin{itemize}}
\newcommand{\ei}{\end{itemize}}
\newcommand{\beq}{\begin{equation}}
\newcommand{\eeq}{\end{equation}}

\newcommand{\ra}{\rightarrow}

\newcommand{\dotcup}{\ensuremath{\mathaccent\cdot\cup}}
\newcommand{\floor}[1]{\lfloor {#1}\rfloor}
\newcommand{\onto}{\twoheadrightarrow}

\newcommand{\mf}[1]{\mathfrak{#1}}

\newcommand{\pF}{\mf{p}_F}

\newcommand{\tr}{\operatorname{tr}}

\newcommand{\trace}{\operatorname{tr}}
\newcommand{\Trace}{\operatorname{Tr}}

\newcommand{\Hom}{\operatorname{Hom}}
\newcommand{\End}{\operatorname{End}}
\newcommand{\Aut}{\operatorname{Aut}}

\newcommand{\im}{\operatorname{im}}

\newcommand{\diag}{\operatorname{diag}}

\newcommand{\proj}{\operatorname{proj}}

\newcommand{\id}{\operatorname{id}}

\newcommand{\GL}{\operatorname{GL}}

\newcommand{\U}{\operatorname{U}}

\newcommand{\bbZ}{\mathbb{Z}}

\newcommand{\bbN}{\mathbb{N}}

\newcommand{\Gal}{\mathop{\text{Gal}}}

\newcommand{\Gcd}[2]{\mathop{\text{gcd}(#1,#2)}}

\newcommand{\nc}{\newcommand}
\nc{\Build}[1]{{\cal B}^1} 
\nc{\build}[1]{{\cal b}} 

\def\B{\mathcal B}
\def\C{\mathcal C}

\title{Intertwining semisimple characters for~$p$-adic classical groups}
\date{\today}
\author{Daniel Skodlerack and Shaun Stevens}

\newcommand{\mfa}{\mf{a}}

\let\pi\varpi
\def\ignore#1{\relax}
\def\floor#1{{\left\lfloor{#1}\right\rfloor}}

\begin{document}
\maketitle
\begin{abstract}
\ignore{
Let~$G$ be a unitary group of an~$\epsilon$-hermitian form~$h$ given over a nonarchimedean local field~$F_0$ of odd residue 
characteristic. We introduce a geometric combinatoric condition under which we prove ''Intertwining implies Conjugacy'' for semisimple characters of~$G$ and the general linear group of the ambient vector space of~$G$. Further we prove a Skolem-Noether result for the action of~$G$ on its Lie algebra, more precisely two Lie algebra elements of~$G$ which have the same characteristic polynomial over~$F$ must be conjugate under an element of~$G$ if there are corresponding semisimple characters which intertwine over an element of~$G$ }
Let~$G$ be a unitary group over a nonarchimedean local field of odd residual characteristic. This paper concerns the study of the ``wild part'' of the irreducible smooth representations of~$G$, encoded in a so-called ``semisimple character''. We prove two fundamental results concerning them, which are crucial steps towards a classification of the cuspidal representations of~$G$. First we introduce a geometric combinatoric condition under which we prove an ''intertwining implies conjugacy'' theorem for semisimple characters, both in~$G$ and in the ambient general linear group. Second, we prove a Skolem--Noether theorem for the action of~$G$ on its Lie algebra; more precisely, two semisimple elements of the Lie algebra of~$G$ which have the same characteristic polynomial must be conjugate under an element of~$G$ if there are corresponding semisimple strata which are intertwined by an element of~$G$.
MSC2010 [11E57] [11E95] [20G05] [22E50]
\end{abstract}


\section{Introduction}
One of the reasons for studying the representation theory of~$p$-adic groups is, via the local Langlands correspondence, to understand Galois representations: eventually, one can hope to get explicit information on Galois representations from this correspondence. The arithmetic core, which is rather mysterious on the Galois side, is encoded in restriction to wild inertia. On the automorphic side, this restriction should correspond to looking at certain representations of certain pro-$p$-subgroups.

For~$p$-adic general linear groups, Bushnell--Kutzko~\cite{bushnellKutzko:93} constructed, and classified, all cuspidal irreducible representations. At the heart of this classification sit the so-called ``simple characters''; these are very particular arithmetically-defined characters of pro-$p$-subgroups, which exhibit remarkable rigidity properties (see below for details). These properties were exploited, and extended, by Bushnell--Henniart~\cite{bushnellHenniart:96}, who defined the notion of an ``endo-class'' and hence proved a Ramification Theorem~\cite{bushnellHenniart:03} for the local Langlands correspondence for general linear groups: there is a bijection between the set of endo-classes and the set of orbits (under the Weil group) of irreducible representations of the wild inertia group. More recently, they have extended this, using the fundamental structural properties of simple characters to prove a Higher Ramification Theorem~\cite{bushnellHenniart:16}.

For~$p$-adic classical groups -- that is, symplectic, special orthogonal and unitary groups -- in odd residual characteristic, analogous characters were constructed by the second author~\cite{stevens:05} as a fundamental step in his construction of all cuspidal irreducible representations~\cite{stevens:08}. This required first extending the theory of simple characters to the case of ``semisimple characters'' (see also the work of Dat~\cite{dat:09}). However, the rigidity results which allowed Bushnell--Kutzko to obtain a \emph{classification} were missing -- partly because some of them are false. 

In this paper, we prove many of these rigidity results for semisimple characters, which are new even in the case of general linear groups -- in particular, we prove ``intertwining implies conjugacy'' and Skolem--Noether results (see below for details). In a sequel~\cite{kurinczukskodlerackstevens:16}
, jointly with Kurinczuk, we are then able to put this together with other work of Kurinczuk and the second author~\cite{kurinczukstevens:15}, to turn the construction of cuspidal representations into a classification, both for complex representations and in the~$\ell$-modular case. 

We anticipate further work to come from these rigidity results. Semisimple characters (or, more precisely, their endo-classes) should give a decomposition of the category of smooth~$\ell$-modular representations of classical groups. The Glauberman correspondence leads to a conjecture of a certain form of Frobenius reciprocity between smooth representations of a classical group and certain self-dual representations of the ambient general linear group, which respects the decomposition parametrized by semisimple characters and allows one to transfer questions about classical groups directly to those about general linear groups. Finally, we expect that all can be generalized to proper inner forms of classical groups, where additional problems (analogous to those in the case of inner forms of general linear groups~\cite{broussousSecherreStevens:12}) arise.

\medskip

Now we state our results more precisely. Let~$F$ be a nonarchimedean local field of odd residual characteristic. Let~$G$ be the isometry group of an~$\epsilon$-hermitian space with respect to some automorphism of~$F$ of order at most two, so that~$G$ is the group of fixed points under an involution on the full automorphism group~$\tilde{G}$ of the underlying~$F$-vector space~$V$. We similarly regard the Lie algebra of~$G$ as the fixed points of an involution on~$A=\End_F(V)$. Note that, when~$\epsilon=1$ and the involution on~$F$ is trivial, we are working with the \emph{full} orthogonal group; however, the set of semisimple characters for the full orthogonal group and for the special orthogonal group coincide.

The starting point in the construction of semisimple characters is an algebraic combinatoric object, a so-called \emph{semisimple stratum}~$[\Lambda,q,r,\beta]$. The principal data here are: an element~$\beta\in A$ which generates a sum of field extensions~$E=F[\beta]=\bigoplus_{i\in I}E_i$; and a rational point~$\Lambda$ in the (enlarged) Bruhat--Tits building of the centraliser of~$\beta$ in~$G$, which we think of as a lattice sequence in~$V$ (see~\cite{broussousLemaire:02}). Associated to~$\Lambda$, we have a filtration~$(\mfa_{n})_{n\in\mathbb Z}$ of~$A$ (which is the Moy--Prasad filtration) and the integer~$q$ is defined by~$\beta\in\mfa_{-q}\setminus\mfa_{1-q}$; this is required to be positive. Finally,~$r$ is an integer between~$0$ and~$q$ which is small enough in the following approximate sense: the stratum~$[\Lambda,q,r,\beta]$ corresponds to the coset~$\beta+\mfa_{-r}$ and~$r$ must be small enough so that the \emph{formal intertwining} of the coset has a nice formula involving the centralizer of~$\beta$. (See section~\ref{secSemisimpleStrata} for more details, and a precise definition.) A semisimple stratum~$[\Lambda,q,r,\beta]$ as above splits according to the primitive idempotents~$1^i$ of~$E$, giving \emph{simple} strata~$[\Lambda^i,q_i,r,\beta_i]$ in~$V^i=1^i V$, which are studied in~\cite{bushnellKutzko:93}. In particular, a semisimple stratum is simple if and only if its indexing set~$I$ has cardinality one.

Associated to any semisimple stratum~$[\Lambda,q,r,\beta]$, and for any integer~$m\ge 0$, we have a family~$\mathcal C(\Lambda,m,\beta)$ of \emph{semisimple characters}. We do not recall the definition here (see again Section~\ref{secSemisimpleStrata}) but note only that, by applying the idempotents, we obtain from a semisimple character~$\theta$ a collection of \emph{simple} characters~$\theta_i$, for~$i\in I$. For simple characters, the fundamental rigidity property proved in~\cite{bushnellKutzko:93} for lattice chains (i.e. sequences without repetition), is the following:
\begin{quote}
Suppose~$\theta\in\mathcal C(\Lambda,m,\beta)$ and~$\theta'\in\mathcal C(\Lambda,m,\beta')$ are \emph{simple} characters which intertwine in~$\tilde G$. Then they are conjugate in the parahoric subgroup~$\tilde{\U}(\Lambda)$.
\end{quote}
In the case of semisimple characters, this result is false as soon as~$|I|>1$: the essential reason is that one can have two lattice sequences (or even chains)~$\Lambda,\Lambda'$ which are conjugate in~$\tilde G$ but such that the separate pieces~$\Lambda^i,\Lambda'^i$ are not (all) conjugate in~$\Aut_F(V^i)$. (Equivalently, there are points in the building of a proper Levi subgroup of~$\tilde G$ which are not conjugate under the Levi but are conjugate under~$\tilde G$.) For similar reasons, the result would remain false if one weakened the conclusion to only conjugacy under~$\tilde G$. Thus one needs an extra condition to ensure that intertwining implies conjugacy. In order to describe this condition, we need a ``matching theorem'' for semisimple characters which intertwine:

\begin{theorem*}[{Theorem~\ref{thmMatchingForChar}}]
Let~$\theta\in \C(\Lambda,m,\beta)$ and~$\theta'\in \C(\Lambda',m,\beta')$ be semisimple characters which intertwine. Then there is a unique bijection~$\zeta$ between the index sets~$I$ and~$I'$ such that the simple characters~$\theta_i$ and~$\theta'_{\zeta(i)}$ are intertwined by an isomorphism in~$\Hom_F(V^i,V'^{\zeta(i)})$.
\end{theorem*}

This matching theorem allows us to describe a condition which is certainly necessary for conjugacy: if~$\theta,\theta'$ as in the Theorem are conjugate by an element of the parahoric subgroup~$\tilde{\U}(\Lambda)$ then, with~$\zeta:I\to I'$ the matching given by the previous Theorem, we have
\begin{equation}\label{eqn:condition}
\dim_{\kappa_F}\Lambda_l^i/\Lambda_{l+1}^i=
\dim_{\kappa_F} \Lambda'^{\zeta(i)}_l /\Lambda'^{\zeta(i)}_{l+1}, \qquad
\text{for all }i\in I\text{ and }l\in\mathbb Z.
\end{equation}
It turns out that this condition is also sufficient to obtain an ``intertwining implies conjugacy'' result:

\begin{theorem*}[{Theorem~\ref{thmintConjSemisimpleCharacters}}]
Let~$\theta\in \C(\Lambda,m,\beta)$ and~$\theta'\in \C(\Lambda',m,\beta')$ be semisimple characters which intertwine, let~$\zeta: I\ra I'$ be the matching given by Theorem~\ref{thmMatchingForChar}, and suppose that the condition~\eqref{eqn:condition} holds. Then~$\theta$ is conjugate to~$\theta'$~by an element of~$\tilde{\U}(\Lambda)$.
\end{theorem*}

\medskip

Now we turn to our results for classical groups, so we assume that our underlying strata~$[\Lambda,q,r,\beta]$ are \emph{self-dual} -- that is,~$\beta$ is in the Lie algebra of~$G$ and~$\Lambda$ is in the building of the centralizer in~$G$ of~$\beta$ (see~\cite{broussousStevens:09}). Our first main result here is a Skolem--Noether theorem for semisimple strata, which is crucial in the sequel~\cite{kurinczukskodlerackstevens:16}:

\begin{theorem*}[{Theorem~\ref{thmSkolemNoetherSemisimple}}]
Let~$[\Lambda,q,r,\beta]$ and~$[\Lambda',q,r,\beta']$ be two self-dual semisimple strata which intertwine in~$G$, and suppose that~$\beta$ and~$\beta'$ have  the same characteristic polynomial. Then, there is an element~$g \in G$ such that~$g\beta g^{-1}=\beta'$.
\end{theorem*}

In order to prove this statement, in Section~\ref{secWittGroups} we analyse the Witt groups~$W_*(E)$ of finite field extensions~$E$ of~$F$ and trace-like maps from~$W_*(E)$ and~$W_*(F)$ .

\medskip

Given a self-dual semisimple stratum~$[\Lambda,q,r,\beta]$, the set~$\C_-(\Lambda,m,\beta)$ of semisimple characters for~$G$ is obtained by restricting the semisimple characters in~$\C(\Lambda,m,\beta)$. (Equivalently, one may just restrict those senisimple characters which are invariant under the involution defining~$G$.) Our final result is an ``intertwining implies conjugacy'' theorem for semisimple characters for~$G$.

\begin{theorem*}[{Theorem~\ref{thmintConjSemisimpleCharactersForG}}]
Let~$\theta_-\in \C_-(\Lambda,m,\beta)$ and~$\theta'_-\in \C_-(\Lambda,m,\beta')$ be two semi\-simple characters of~$G$, which intertwine over~$G$, and assume that their matching satisfies~\eqref{eqn:condition}. Then,~$\theta_-$ and~$\theta'_-$ are conjugate under~$\U(\Lambda)=\tilde\U(\Lambda)\cap G$. 
\end{theorem*}

\section{Notation}
Let~$(F,\nu_F)$ be a nonarchimedean local field of odd residual characteristic, equipped with an involution~$\rho$ (which may be trivial) with fixed field~$F_0$. We write~$o_F, \mf{p}_F$ and~$\kappa_F$ for the valuation ring, its maximal ideal and the residue field of~$F$ respectively, and we assume that the image of the additive valuation~$\nu:=\nu_F$ is~$\bbZ\cup\{\infty\}$. We also denote by~$x\mapsto \bar{x}$ the reduction map~$o_F\onto\kappa_F=o_F/\mf{p}_F$. We fix a symmetric or skew-symmetric uniformizer~$\pi\in \mf{p}_F\setminus\mf{p}_F^2$. We use similar notation for other nonarchimedean local fields. If~$E|F$ is an algebraic field extension then we write
~$E^{ur}$ for the 
maximal unramified subextension of~$E|F$.

Let~$h$ be an~$\epsilon$-hermitian form (with~$\epsilon=\pm 1$) on an~$F$-vector space~$V$ of finite dimension~$m$,
i.e.\ for all~$v_1,v_2\in V$ and~$x,y\in F$ the bi-additive form~$h$ satisfies 
\[
h(v_1 x,v_2 y)=\rho(x)\epsilon \rho(h(v_2,v_1))y.
\]
We denote the ring of~$F$-endomorphisms of~$V$ by~$A$ and its group of units~$A^\times$ by~$\tilde{G}$.  Let~$G$ be the group of all elements~$g$ of~$\tilde{G}$ such that 
$h(gv_1,gv_2)$ is equal to~$h(v_1,v_2)$, for all vectors~$v_1,v_2$; this is the group of points of a reductive group over~$F_0$, which is connected unless~$F=F_0$ and~$\epsilon=+1$, in which case it is the full orthogonal group. Let~$\sigma=\sigma_h$ be the adjoint anti-involution of~$h$ on~$A$. For a~$\sigma$-stable subset~$M$ of~$A$, we write~$M_+$ for the set of symmetric elements and~$M_-$ for the set of skew-symmetric elements.

An~$o_F$-lattice in~$V$ is a free~$o_F$-module~$M$ of dimension~$m$. The dual~$M^\#$ of~$M$ with respect to~$h$ is the set of all vectors~$v$ of~$V$ such that~$h(v,M)$ is a subset of~$\pF$. A \emph{lattice sequence} in~$V$ is a map~$\Lambda$ from~$\bbZ$ to the set of~$o_F$-lattices of~$V$
satisfying 
\begin{enumerate}
\item~$\Lambda_s\subseteq\Lambda_t$, for all integers~$s>t$, and
\item~$\Lambda_s\pi =\Lambda_{s+e}$ for some (unique) integer~$e$ and all integers~$s$.
\end{enumerate} 
We call~$e=:e(\Lambda| o_F)$  the~$o_F$-period of~$\Lambda$. An injective lattice sequence is called a \emph{lattice chain}. For each integer~$s$, we denote by~$x\mapsto \bar{x}$ the reduction map~$\Lambda_s\onto \Lambda_s/\Lambda_{s+1}$. A lattice sequence~$\Lambda$ is called \emph{self-dual} 
if there is an integer~$u$ such that~$(\Lambda_s)^{\#}=\Lambda_{u-s}$.  

As usual, a lattice sequence~$\Lambda$ determines the following filtrations of~$A$ and~$A_-$ (if~$\Lambda$ is self-dual):~$\mf{a}_i(\Lambda)$ is the set of all elements of~$A$ which map~$\Lambda_s$ into~$\Lambda_{s+i}$ for all integers~$s$ and~$\mf{a}_{-,i}(\Lambda)$ is the intersection of 
$\mfa_i(\Lambda)$ with~$A_-$.
We skip the argument~$\Lambda$ if there is no cause of confusion and we write~$\mfa'_i$ if there is a second lattice sequence~$\Lambda'$ given. 

The sequence~$\Lambda$ also induces filtrations on~$\tilde{\U}(\Lambda):=\mf{a}_0^\times$ by~$\tilde{\U}^i(\Lambda)=1+\mf{a}_i$ and, when~$\Lambda$ is self-dual, on~$\U(\Lambda):=\tilde{\U}(\Lambda)\cap G$ by~$\U^i(\Lambda)=G\cap \tilde{\U}^i(\Lambda)$ for~$i\in\bbN$. The filtration on~$A$ defines a ``valuation map''~$\nu_\Lambda$ as follows: for~$\beta\in A$, we put~$\nu_{\Lambda}(\beta)=\sup\{i\mid\beta\in\mf{a}_i\}$, an integer or~$\infty$. The normalizer~$\mf{n}(\Lambda)$ of~$\Lambda$ is the set of elements of~$g\in A^\times$ such that~$\nu_\Lambda(g^{-1})=-\nu_\Lambda(g)$. 

The \emph{translation} of~$\Lambda$ by~$s\in \bbZ$ is the lattice 
sequence~$(\Lambda+s)_i:=\Lambda_{i-s}$. The lattice sequence 
\[
\Lambda\oplus(\Lambda+1)\oplus\cdots\oplus(\Lambda+e(\Lambda|o_F)-1)
\] 
is always a lattice chain. 
By this construction, many theorems in~\cite{bushnellKutzko:93} proven for lattice chains are valid for lattice sequences (cf.~\cite{stevens:05}, and also~\cite{kurinczukstevens:15}, where this is called a~$\dag$-construction). If this is the case, or the proof of a result for lattice chains is valid for lattice sequences without change, then, in the following, we just refer to the statement for lattice chains. 

Finally, for~$x$ a real number, we denote by~$\floor{x}$ the greatest integer not greater than~$x$.


\section{Lifting isometries}
The isomorphism type of the hermitian space~$(V,h)$ is encoded in any self-dual lattice sequence of~$V$, as explained in this section. The main results are Proposition~\ref{propliftResidualIsometries} and Corollary~\ref{corCohomologyArgument}, which explain how an approximate isometry (for example, one which induces an isometry at the level of residue fields) can be lifted to a genuine isometry. Let us state the main proposition:

\begin{proposition}\label{propliftResidualIsometries}
Let~$F|F'$ be a finite field extension.  Suppose we are given two finite-dimensional~$\epsilon$-hermitian spaces~$(V,h)$ and~$(V',h')$ with respect to~$(F,\rho)$, an~$F'$-linear isomorphism~$f:V\ra V'$ and two self-dual~$o_F$-lattice sequences~$\Lambda$ and~$\Lambda'$ of 
$(V,h)$ and~$(V',h')$, respectively, such that, for all~$i\in\bbZ$,
\bi
\item~$f(\Lambda_i)=\Lambda'_i$,
\item~$f((\Lambda_i)^{\#})=f(\Lambda_i)^{\#}$,
\item~$\overline{h'(f(v),f(w))}=\overline{h(v,w)}\in\kappa_F$, for all~$v\in \Lambda_i$,~$w\in (\Lambda_{i+1})^{\#}$and
\item~$\overline{f(vx)}=\overline{f(v)x}\in\Lambda'_{i+e(\Lambda'|o_F)\nu_k(x)}/\Lambda'_{1+i+e(\Lambda'|o_F)\nu_k(x)}$, for all~$v\in \Lambda_i$,~$x\in F^{\times}$.
\ei
Then
there is an~$F$-linear isometry~$g$ from~$(V,h)$ to~$(V',h')$ mapping~$\Lambda$ to~$\Lambda'$
such that~$(f-g)(\Lambda_i)\subseteq \Lambda'_{i+1}$ for all integers~$i$.
\end{proposition}

Later it will be useful to have a stronger approximation statement. For that we introduce a generalization of the adjoint anti-involution. 
For two finite-dimensional~$\epsilon$-hermitian spaces~$(V,h)$ and~$(V',h')$ with respect to~$(F,\rho)$ there is a map 
$\sigma_{h,h'}$ from~$\Hom_F(V,V')$ to~$\Hom_F(V',V)$
defined, for~$f\in\Hom_F(V,V')$, by the equation 
\[
h'(f(v),w)=h(v,\sigma_{h,h'}(f)(w)),\qquad\text{for }v\in V,\ w\in V'.
\]

\begin{corollary}\label{corCohomologyArgument}
Let~$n$ be a positive integer, let~$\Lambda$ be a self-dual lattice sequence and let~$f: V \rightarrow V'$ be an~$F$-linear isomorphism such that~$f(\Lambda)$ is self-dual and
\[
\sigma_{h,h'}(f)\in \tilde{\U}^n(\Lambda)f^{-1}\tilde{\U}^n(f(\Lambda)).
\]
Then there is an isomorphism~$\phi$ of~$\epsilon$-hermitian spaces from~$(V,h)$ to~$(V',h')$ contained in~$\tilde{\U}^n(f(\Lambda))f\tilde{\U}^n(\Lambda)$.
\end{corollary}

\begin{proof}
The~$\epsilon$-hermitian spaces~$(V,h)$ and~$(V',h')$ are isomorphic to each other by an isomorphism which maps~$\Lambda$ to~$f(\Lambda)$, by Proposition~\ref{propliftResidualIsometries}. Thus we can restrict to the case where~$(V,h)$ is equal to~$(V',h')$ and~$f(\Lambda)$ is equal to~$\Lambda$. By assumption, the double coset 
$\tilde{\U}^n(\Lambda)f\tilde{\U}^n(\Lambda)$ is invariant under the automorphism~$g\mapsto \sigma_h(g^{-1})$, and this double coset thus has a fixed point, by~\cite[Lemma~2.2]{stevens:01}. 
\end{proof}

We need a sequence of lemmas to prove Proposition~\ref{propliftResidualIsometries}.

\begin{lemma}\label{lemLiftWitt1}
Suppose that~$\Lambda$ is a self-dual lattice chain of period~$1$ such that~$\Lambda_0^{\#}=\Lambda_1$. Consider the form
\[
\bar{h}: \Lambda_0/\Lambda_1\times   \Lambda_0/\Lambda_1\ra \kappa_F 
\]
defined by~$\bar{h}(\bar{v},\bar{w})=\overline{h(v,w)}.$
Then every Witt basis of~$(\Lambda_0/\Lambda_1,\bar{h})$ lifts to a Witt basis of~$(V,h)$ contained in~$\Lambda_0$, under the projection~$\Lambda_0\onto \Lambda_0/\Lambda_1$.  
\end{lemma}

\begin{proof}
Let~$\B$ be a Witt basis of~$\bar{h}$. We have 
\[
\B=\B_0\dotcup \B_{1,-1}\dotcup \B_{2,-2}\dotcup\ldots\dotcup \B_{r,-r},
\]
where~$\B_{i,-i}$ spans a hyperbolic space,~$\B_0$ spans an anisotropic space, and all these spaces are pairwise orthogonal to each other in~$\Lambda_0/\Lambda_1$. Further we have a decomposition 
\[
\B_0=\B_{0,1}\dotcup \B_{0,2}\dotcup\ldots\dotcup \B_{0,t}
\]
into pairwise orthogonal sets of cardinality one. Take an arbitrary lift~$\B'^{(0)}$ of~$\B$ to~$\Lambda_0$; for an element~$\bar{v}\in\B$, we write~$v\in\B'^{(0)}$ for its lift. 

\textbf{Step 1:} Consider~$\B_{0,1}=\{\bar{v}_0\}$; put~$W:=v_0^{\perp}$ and define
\[
\B'^{(1)}:=\{\proj_W(v)\mid v\in \B'^{(0)}\setminus\{v_0\}\}\cup \{v_0\},
\] 
where~$\proj_W$ denotes the orthogonal projection onto~$W$. We recall the formula
\[
\proj_W(v)=v-v_0\frac{h(v_0,v)}{h(v_0,v_0)}
\]
and conclude that~$\overline{h(\proj_W(v),\proj_W(v'))}$ is equal to
\[
\overline{h(v,v')}-\overline{\frac{h(v_0,v')h(v,v_0)}{h(v_0,v_0)}}-\overline{\frac{\rho(h(v_0,v))h(v_0,v')}{\rho(h(v_0,v_0))}}+\overline{\frac{\rho(h(v_0,v))h(v_0,v')}{\rho(h(v_0,v_0))}}
\]
and therefore equal to~$\overline{h(v,v')}$ for all~$v,v'\in B'^{(0)}.$ Thus, replacing~$(V,h)$ by~$(W,h_{|W})$ and~$\Lambda$ by its intersection with~$W$ and then repeating, we can assume that~$\B_0$ is empty.

\textbf{Step 2:} Consider~$\B_{1,-1}=\{\bar{v}_1,\bar{v}_{-1}\}$ and define now~$W:=\{v_1,v_{-1}\}^\perp$.
Then, as in Step~1, elements~$v$ and~$v'$ of 
\[
\B'^{(1)}:=\{\proj_W(v)\mid v\in \B'^{(0)}\setminus\{v_1,v_{-1}\}\}\cup\{v_1,v_{-1}\}
\] 
satisfy 
$\overline{h(\proj_w(v),\proj_W(v'))}=\overline{h(v,v')}$, because if~$v\in\B'^{(0)}\setminus\{v_1,v_{-1}\}$ then 
\[
v\equiv \proj_W(v) + v_{-1} h(v_1,v) + v_1 \epsilon h(v_{-1},v) \pmod{\Lambda_1}.
\] 
Thus we have reduced to the hyperbolic case that~$\B$ is equal to~$\B_{1,-1}$.

\textbf{Step 3:}~$\B=\B_{1,-1}=\{\bar{v}_1,\bar{v}_{-1}\}$.
The sequence~$(w_i)_{i\ge 1}$ ,defined by~$w_1:=v_1$
and 
\[
w_{i+1}:=w_i-v_{-1}\frac{h(w_i,w_i)}{2}, \qquad\text{for }i\ge 1,
 \]
has a limit~$v_1'$ which satisfies~$h(v'_1,v'_1)=0$ and~$\bar{v}'_1=\bar{v}_1$, and analogously we find~$v'_{-1}$
with similar properties. Then 
\[
\B'^{(1)}:=\left\{ \frac{1}{ v'_1 \rho(h(v'_1,v'_{-1}))} , v'_{-1}\right\}
\]
is a Witt basis of~$V$ which lifts~$\B$.
\end{proof}

\begin{lemma}\label{lemLiftWitt2}
Suppose that~$\Lambda$ is a self-dual lattice chain of period~$1$ such that~$\Lambda_0^{\#}=\Lambda_0$. Consider the form
\[
\bar{h}:  \Lambda_0/\Lambda_1\times  \Lambda_0/\Lambda_1\ra \kappa 
\]
defined by~$\bar{h}(\bar{v},\bar{w})=\overline{h(v,w)\pi^{-1}}.$
Then for every Witt basis~$\B=\B_0\dotcup \B^-\dotcup \B^+$  of~$(\Lambda_0/\Lambda_1,\bar{h})$,  with isotropic parts~$\B^-$ and~$\B^+$ and anisotropic part~$\B_0$, there is a Witt basis~$\B'=\B'_0\dotcup \B'^+\dotcup \B'^-\subseteq \Lambda_{-1}$ of~$(V,h)$ such that 
$\B'_0,\  \B'^+$ and~$\B'^-\pi$  are lifts of~$\B_0,\ \B^+$ and~$\B^-$ under the projection 
$\Lambda_0\onto \Lambda_0/\Lambda_1$, respectively.
\end{lemma}

Here we explicitly make use of the fact that~$\rho(\pi)\in\{\pi,-\pi\}$. 

\begin{proof}
This follows directly from Lemma~\ref{lemLiftWitt1} if we substitute~$h$ by~$h\pi^{-1}$.
\end{proof}

We need a third base case for period~$2$.

\begin{lemma}\label{lemLiftWitt3}
Suppose that~$\Lambda$ is a self-dual lattice chain of period~$2$ such that~$\Lambda_0^{\#}=\Lambda_0$. Then~$h$ has anisotropic dimension zero 
and for any basis~$\B_{0}$ of~$\Lambda_0/\Lambda_{1}$ there is a Witt basis for~$h$,
\[
\B'= \B'_{-1}\cup \B'_{0},
\]
such that~$\B'_i$ is a subset of~$\Lambda_i\setminus\Lambda_{i+1}$ for all~$i$ and such that~$\B'_{0}$ is a lift of~$\B_{0}$ under the projection~$\Lambda_0\onto \Lambda_0/\Lambda_1$. 
Further,~$h$ vanishes on~$\B'_0\times \B'_0 $.
\end{lemma}

\begin{proof}
First we prove that~$h$ is hyperbolic. Suppose for contradiction that it has positive anisotropic dimension, i.e. let~$v$ be an anisotropic vector and part of a Witt basis for~$h$ which splits~$\Lambda$. We can multiply~$v$ by a scalar such that~$h(v,v)$ is a unit or a uniformizer of $F$. We treat only  the second case, because the first one is similar. There is an index~$i$ such that~$\Lambda_i\cap vF$ is equal to~$vo_F$, and then this is equal to~$\Lambda_i\cap vF$ because~$h(v,v)$ is uniformizer.  Since, for all lattices in the image of~$\Lambda$ the homothety class is invariant under dualization, we obtain that the index has to be zero. Thus,~$\Lambda_{-1}\cap vF=\mf{p}_F^{-1}$ is equal to~$\pi^{-2}(\Lambda_{1}\cap vF)$, which is a contradiction.  

Now let us construct the lift. We start with a Witt basis~$\B''$ for~$h$ which splits~$\Lambda$. Let~$\B''_0$ be the set of elements~$v$ of~$\B''$ such that
\[
vF\cap\Lambda_0\neq vF\cap\Lambda_{1},
\]
and let~$W_0$ be the span of~$\B''_0$. We prove that the restriction of~$h$ to~$W_0$ is zero. We define, for~$v\in\B''$, the element~$v^*$ to be the element of~$\B''$ such that~$h(v,v^*)$ is non-zero, i.e. equal to~$1$ or~$-1$.
If there is an element~$v\in W_0\cap \B''$ such that~$v^*\in W_0$ then~$\Lambda_{-1}\cap (vF+v^*F)=\Lambda_{0}\cap (vF+v^*F)$ and thus this coincides with
~$(\Lambda_{-1})^\# \cap (vF+v^*F)$. This is a contradiction because~$(\Lambda_{-1})^\#$ is equal to~$\Lambda_{1}$. This shows that~$h$ is zero on~$W_0$. 
Thus, multiplying elements of~$\B''_0$ by scalars if necessary, we can assume that $\B''_0$ is a subset of~$\Lambda_0\setminus\Lambda_1$. 
By the definition of~$W_0$ we have that, for all~$v\in\B''_0$, the intersection of~$vF$ with~$\Lambda_{-1}$ is~$vo_F$ for all~$v\in\B''_0$ and thus taking duals  we get
that the intersection of~$v^*F$ with~$\Lambda_{1}$ is~$v^* \mf{p}_F$, and thus~$\B''$ is a subset of~$\Lambda_{-1}\setminus\Lambda_0$.
Thus, we have now found a basis~$\B''$ satisfying all the conditions except that~$\B''_0$ need not be a lift of~$\B_0$. Now a base change 
from~$\B''_0$ to a lift of~$\B_0$ in~$W_0$, together with the adjoint base change on the span of~$\B'' \setminus\B''_0$, finishes the proof.  
\end{proof}

\begin{corollary}\label{corLemWittLift3}
 Under the assumptions of Lemma \ref{lemLiftWitt3} there is a unique~$\kappa$-basis~$\B_{-1}$ of~$\Lambda_{-1}/\Lambda_0$ such that, 
 for all elements~$x$ of~$\B_0$, there is exactly one element~$y$ of~$\B_{-1}$ such that
 \[\bar{h}(y,z)\left\{ \begin{array}{ll}1&, \text{ if }z=x \\0&, \text{ if }z\in\B_0\setminus\{x\}\end{array}\right.\]
 where~$\bar{h}: \Lambda_{-1}/\Lambda_0 \times \Lambda_0/\Lambda_1\rightarrow \kappa$ is the form induced from~$h$.
 Further, there is a Witt basis for~$h$ which lifts~$\B_0\cup\B_{-1}$. 
\end{corollary}

\begin{proof}
 By Lemma \ref{lemLiftWitt3} the form~$\bar{h}$ is non-degenerate and thus identifies the dual of~$\Lambda_0/\Lambda_1$ with~$\Lambda_{-1}/\Lambda_0$ with~$\sigma$-twisted
~$\kappa$-action. We take for~$\B_{-1}$ the basis dual to~$\B_0$. The remaining part follows from Lemma \ref{lemLiftWitt3}. 
\end{proof}

We put together the two previous results to treat the general case.

\begin{lemma}\label{lemWittLift}
Let~$\Lambda$ be a self-dual lattice chain of period~$e$ and let~$\B$ be a subset of~$V$ satisfying the following conditions:
\be
\item~$(\Lambda_0)^\#\in\{\Lambda_0,\Lambda_1\}$;
\item~$\B=\bigcup_{i=\floor{\frac{1-e}{2}}}^{\floor{\frac{e-1}{2}}} \B_i$, with~$\B_i\subseteq\Lambda_i\setminus \Lambda_{i+1}$;
\item~$\bar{\B}_i$, the image of~$\B_i$ in~$\Lambda_i/\Lambda_{i+1}$, is a basis of~$\Lambda_i/\Lambda_{i+1}$;
\item for all~$i\in \{0,1,\ldots ,\floor{\frac{e-1}{2}}\}$ with~$(\Lambda_i)^\#\not\in\{\Lambda_{i+1},\Lambda_{i+1-e}\}$ and all~$v\in \B_i$ there exists a unique~$v'\in \B\cap (\Lambda_{i+1})^\#\setminus (\Lambda_i)^\#$ such that 
$\overline{h(v,v')}=\bar{1}$;
\item  if~$(\Lambda_0)^\#=\Lambda_1$ then~$\bar{\B}_0$ is a Witt basis for~$(\Lambda_0/\Lambda_1,\bar{h})$;
\item if~$(\Lambda_{\floor{\frac{e-1}{2}}})^\#=\Lambda_{\floor{\frac{1-e}{2}}}$ then
$\bar{B}_{\floor{\frac{e-1}{2}}}$ is a Witt basis of~$(\Lambda_{\floor{\frac{e-1}{2}}}/\Lambda_{
{\frac{e+1}{2}}},\overline{h\pi^{-1}})$.
\ee
Then there is a basis~$\B'$ of~$(V,h)$ such that
\bi
\item[(a)]~$\B'=\bigcup_{i=\floor{\frac{1-e}{2}}}^{\floor{\frac{e-1}{2}}} \B'_i$, where~$\B'_i:=\B'\cap(\Lambda_i\setminus \Lambda_{i+1})$, for all~$i$,
\item[(b)]~$\bar{\B'}_i=\bar{\B}_i$, for all~$i$, and
\item[(c)]~$\B'$ is a Witt basis of~$(V,h)$ up to multiplication of some isotropic elements of~$\B'_{\floor{\frac{e-1}{2}}}$ by~$\pi^{-1}$. 
\ei
\end{lemma}

\begin{proof}
The lattice chain~$\Lambda$ is split by a Witt decomposition; that is, there are pairwise orthogonal~$\epsilon$-hermitian spaces~$V^i,\ i\in\{0,\ldots,\floor{\frac{e-1}{2}}\}$ whose sum is~$V$ such that 
\[
(V^i\cap\Lambda_i)+\Lambda_{i+1}=\Lambda_i\text{ and }V^i\cap (\Lambda_{i+1})^\#+(\Lambda_i)^\#=(\Lambda_{i+1})^\#.
\]
Counting dimensions we deduce that~$V^i\cap\Lambda_j$ is a subset of~$ \Lambda_{j+1}$ for all~$j$ with 
$(\Lambda_j)^\#\not\in\{\Lambda_i a, (\Lambda_{i+1})^\# a\mid a\in F^\times\}$.
Now consider, for~$1\le i\le \floor{\tfrac{e-1}{2}}$,
\[
\tilde{\B_i}:=\left\{\proj_{V^i}(v)\mid  v\in \B_j \text{ and }\Lambda_j\in\{\Lambda_i,(\Lambda_{i+1})^\#\}\right\}.
\]
For each~$i$, the lattice sequence~$\Lambda\cap V^i$ in~$V^i$ is a multiple of a lattice chain of period~$1$. Thus, after scaling, we can apply Lemma~\ref{lemLiftWitt1} or~\ref{lemLiftWitt2} or Corollary \ref{corLemWittLift3} on~$(V^i,\Lambda \cap V^i,\bar{\tilde{\B_i}})$ to obtain~$\B'_i$. 
\end{proof}


\begin{proof}[Proof of Proposition~\ref{propliftResidualIsometries}]
We only have to prove that we can replace~$f$ by an~$F$-linear isomorphism, i.e. that we can reduce the argument to~$F=F'$. The rest follows directly from Lemma~\ref{lemWittLift}.

Since the statement depends only on~$\im(\Lambda)$, without loss of generality assume that~$\Lambda$, and therefore~$\Lambda'$ also, is a chain. 
Take a~$\kappa_F$-basis~$(\bar{v}_{ij})_j$ of~$\Lambda_i/\Lambda_{i+1}$ and lift it to~$(v_{ij})_j$, for~$i=0,\ldots,e(\Lambda|o_F)$. Then~$(v_{ij})_{ij}$ is an~$F$-splitting basis of~$\Lambda$. 
Similarly we choose a lift~$(w_{ij})_{ij}$ for~$(\overline{f(v_{ij})})_{ij}$. The~$F$-linear map~$\tilde{f}$
which maps~$v_{ij}$ to~$w_{ij}$ satisfies the assumptions of the Proposition and 
$(f-\tilde{f})(\Lambda_i)\subseteq\Lambda_{i+1}$, for all~$i\in\bbZ$. Thus we can replace~$f$ by~$\tilde{f}$.
\end{proof}


\section{Witt groups}\label{secWittGroups}
In this section we fix a finite field extension~$E|F$ and an involution~$\rho'$ extending~$\rho$. We fix a non-zero~$\rho'$-$\rho$-equivariant~$F$-linear map 
\[
\lambda: E\ra F.
\]
We heavily use in this section that the residue characteristic of~$F$ is odd. 
We will see that the map~$\lambda$ induces in a natural way a map from the Witt group~$W_{\rho',\epsilon}(E)$ of~$(\rho',\epsilon)$-hermitian forms over~$E$ to the Witt group~$W_{\rho,\epsilon}(F)$.

We recall that the Witt group~$W_{\rho,\epsilon}(F)$ is the set of equivalence classes of~$(\rho,\epsilon)$-hermitian forms over~$F$, where we say two such forms are equivalent if their maximal anisotropic direct summands are isometric. We write~$\langle h\rangle$ for the class in~$W_{\rho,\epsilon}(F)$ of signed forms equivalent to~$h$; similarly, for a (skew-) symmetric matrix~$M$, we write~$\langle M\rangle$ for the class of signed hermitian forms equivalent to the form with Gram-matrix~$M$ under the standard basis. 

The group structure on~$W_{\rho,\epsilon}(F)$ is induced by the orthogonal sum. 
Let us recall its structure:

\begin{theorem}
The Witt group~$W_{\rho,\epsilon}(F)$ is isomorphic to 
\begin{enumerate}
\item the trivial group if~$\rho$ is trivial and~$\epsilon=-1$;
\item~$C_2\times C_2$ if~$-1\in (F^\times)^2$ and~$\rho$ is non-trivial;
\item~$C_4$ if~$\rho$ non-trivial and~$-1\not\in (F^\times)^2$;
\item~$C_2\times C_2\times C_2\times C_2$ if~$-1\in (F^\times)^2,$~$\epsilon=1$ and~$\rho$ is trivial;
\item~$C_4\times C_4$ if~$-1\not\in (F^\times)^2,$~$\epsilon=1$ and~$\rho$ is trivial.
\end{enumerate}
\end{theorem}

\begin{proof}
The proof is an easy conclusion of the classification of the hermitian forms using Witt bases, given for example in \cite[1.14]{bruhatTitsIV:87}, and is left to the reader. 
\end{proof}

%
%
%

When it is non-trivial, the group~$W_{\rho,\epsilon}(F)$ is generated by the classes of one-dimensional anisotropic spaces. For example, if~$\epsilon=1$ then: in the case~$F\ne F_0$, the one-dimensional anisotropic spaces are~$\langle(1)\rangle$ and~$\langle(\delta)\rangle$, with~$\delta\in F_0^\times\setminus N_{F/F_0}(F^\times)$; in the case~$F=F_0$, the one-dimensional anisotropic spaces are~$\langle(1)\rangle$,~$\langle(\pi)\rangle$,~$\langle(\delta)\rangle$ and~$\langle(\delta\pi)\rangle$, with~$\delta$ a non-square unit in~$o_F$.

\begin{definition}
We define~$\Trace_{\lambda,\epsilon}$ from~$W_{\rho',\epsilon}(E)$ to~$W_{\rho,\epsilon}(F)$ by
\[
\langle\tilde{h}\rangle\mapsto \langle\lambda\circ \tilde{h}\rangle=:\Trace_{\lambda,\epsilon}(\langle\tilde{h}\rangle).
\]
If~$E|F$ is tamely ramified and~$\lambda=\trace_{E|F}$ then we write~$\Trace_{E|F,\rho',\epsilon}$ for~$\Trace_{\lambda,\epsilon}$.
\end{definition}

For the remainder of the section we often skip the subscripts in~$\Trace$.
\begin{example}
In general, the map~$\Trace_{E|F,\rho',\epsilon}$ is not injective, even if~$\epsilon=1$. For example consider~$E=\mathbb{Q}_3(\sqrt{3},\sqrt{5})$,~$F=\mathbb{Q}_3(\sqrt{5})$ and 
$\rho'(\sqrt{5})=-\sqrt{5}$.
Then 
\[
\Trace_{E|F,\rho',\epsilon}(\langle(\sqrt{3})\rangle)=\left\langle\begin{pmatrix}0& 6\\ 6& 0\end{pmatrix}\right\rangle=0,
\]
so that~$\Trace_{E|F,\rho',\epsilon}$ is not injective. On the other hand, we have that 
\[
\Trace_{E|F,\rho',\epsilon}(\langle(1)\rangle)=\left\langle\begin{pmatrix}2& 0\\ 0& 6\end{pmatrix}\right\rangle\neq 0,
\]
In particular,~$\Trace_{E|F,\rho',\epsilon}$ maps the class~$\langle (\sqrt 3)\oplus(1)\rangle$ of maximal anisotropic dimension to the class~$W_{\rho,\epsilon}(F)$ of maximal anisotropic dimension. We will see that this is always the case.
\end{example}

There is a unique element~$X$ in~$W_{\rho,\epsilon}(F)$ with maximal anisotropic dimension, which we denote by~$X_{\rho,\epsilon,F}$. The main result of this section is the following theorem:

\begin{theorem}\label{thmMaximlaAnisotropicElement}
$
\Trace_\lambda(X_{\rho',\epsilon,E})=X_{\rho,\epsilon,F}.
$
\end{theorem}

The following definition will be useful both in the proof of Theorem~\ref{thmMaximlaAnisotropicElement} and in several other proofs later.

\begin{definition}\label{def:twist}
Let~$\gamma$ be a (skew-)symmetric element of~$\Aut_F(V).$
We define the signed hermitian form 
\[
h_{\gamma}: V\times V\ra F
\]
via 
\[
h_{\gamma}(v,w):=h(v,\gamma w),\ v,w\in V.
\]
We call~$h_\gamma$ the \emph{(skew-) symmetric twist of~$h$ by~$\gamma$}. 
\end{definition}

Note that, if~$h$ is an~$\epsilon$-hermitian form, then~$h_\gamma$ is~$\epsilon$-hermitian when~$\gamma$ is symmetric, and~$(-\epsilon)$-hermitian when~$\gamma$ is skew-symmetric. Twisting by a symmetric element~$\gamma$ induces a permutation of~$W_{\rho,\epsilon}(F)$ and we observe that, by an easy check, the only classes in~$W_{\rho,\epsilon}(F)$ which are preserved by every symmetric twist are the trivial class and the class~$X_{\rho,\epsilon,F}$ of maximal anisotropic dimension. Indeed, twisting by all symmetric elements gives a transitive action on the classes of spaces of fixed odd (anisotropic) dimension.

\begin{proposition}\label{propBijectivityForEFodd}
If~$E|F$ has odd degree, then,~$\Trace_{\lambda}$ is injective.
\end{proposition}

\begin{proof}
There is nothing to say in the symplectic case, so we assume~$\epsilon=1$ or~$F\neq F_0$. Moreover, we can assume that~$\epsilon=1$ because, if~$F\neq F_0$ then a twist by a skew-symmetric element of~$F^\times$ induces bijections~$W_{\rho',1}(E)\to W_{\rho',-1}(E)$ and~$W_{\rho,1}(F)\to W_{\rho,-1}(F)$, commuting with~$\Trace_{\lambda}$. Now~$\Trace_\lambda(\langle(1)\rangle)$ is a class of odd anisotropic dimension, all classes of this anisotropic dimension are symmetric twists of~$\Trace_\lambda(\langle(1)\rangle)$ and they generate~$W_{\rho,1}(F)$. Thus~$\Trace_\lambda$ is surjective and, moreover, bijective, since~$W_{\rho,1}(F)$ is isomorphic to~$W_{\rho',1}(E)$ as groups. 
\end{proof}

\begin{lemma}\label{lemEFDeg2}
Suppose~$E|F$ is of degree~$2$ and~$F=F_0$. Then~$\im(\Trace_{E|F,\rho',1})$ has at least four elements. Further:
\begin{enumerate}
 \item\label{lemEFDeg2.i} If~$E\neq E_0$ then~$\Trace_{E|F,\rho',1}$ is injective.
 \item If~$E=E_0$ then the kernel of~$\Trace_{E|F,\rho',1}$ has exactly~$4$ elements and they have anisotropic dimension at most~$2$. 
\end{enumerate} 
\end{lemma}

\begin{proof}
Take an element~$\delta\in E_0^\times$ and a uniformizer~$\alpha$ of~$E$ which is skew-symmetric with respect to the generator~$\tau$ of~$\Gal(E|F)$.
 Then~$\Trace_{E|F,\rho',1}(\langle\delta\rangle)$ has Gram matrix 
  \[
 \begin{pmatrix} 
 \delta+\tau(\delta)& \alpha(\delta-\tau(\delta))\\ 
 \pm \alpha(\delta-\tau(\delta))& \pm \alpha^2 (\delta+\tau(\delta))
 \end{pmatrix}
 \]
with respect to the~$F$-basis~$\{1,\alpha\}$, where we have~$+$ if~$\rho'$ is trivial and~$-$ if not. 
Its determinant is~$d:=\pm 4\alpha^2 N_{E|F}(\delta)$ and we only have to choose~$\delta$ such that~$-d$ is not a square in~$F$ to get that~$\Trace_{E|F,\rho',1}(\langle(\delta)\rangle)$ is 
non-zero. 

If~$-1\in (E^\times)^2$ then, since~$p$ is odd, also~$-1\in N_{E|F}(E^\times)$ and thus we can find~$\delta\in E_0^\times$ such that~$-d=4\alpha^2$; this is not a square in~$F^\times$ because 
$\alpha\not\in F$. If~$-1\not\in (E^\times)^2$ then~$E|F$ is ramified and~$\nu_F(\alpha^2)=1$ so we can take~$\delta=1$ to get~$-d\not\in (F^\times)^2$.

In either case, we have that~$\Trace_{E|F,\rho',1}(\langle(\delta)\rangle)$ is non-zero for a suitable~$\delta$, and thus of anisotropic dimension~$2$. Taking symmetric twists of~$\Trace_{E|F,\rho',1}(\langle(\delta)\rangle)$ by elements of~$F$ (which commute with~$\Trace_{E|F,\rho',1}$), we see that the image of~$\Trace_{E|F,\rho',1}$ has at least two non-trivial elements and thus, as a subgroup of a~$2$-group, at least four elements in total. This also shows~\ref{lemEFDeg2.i}. 

We consider now the case~$E=E_0$. Take~$y\in o_F$ to be a non-square unit if~$E|F$ is ramified and a uniformizer of~$F$ if~$E|F$ is unramified. Then~$(\alpha)$ and~$(y\alpha)$ are not isomorphic and both are in the kernel of~$\Trace_{E|F,\id,1}$. Since the kernel consists of at most four elements, it is the subgroup generated by~$\langle(\alpha)\rangle$ and~$\langle(y\alpha)\rangle$, which is of order four and consists of classes of spaces of anisotropic dimensions~$0,1,1,2$.
\end{proof}

\begin{proof}[Proof of Theorem~\ref{thmMaximlaAnisotropicElement}]
As in the proof of Proposition~\ref{propBijectivityForEFodd}, we may assume that~$\epsilon=1$.
We only need to prove the statement for one~$\lambda$, because given two such maps~$\lambda_1,\lambda_2$ there is a symmetric element~$z$ of~$E$ such that 
$\lambda_1(zx)=\lambda_2(x)$ for all~$x\in E$. (We thank R.~Kurinczuk for pointing this out.)
Moreover, we only have to prove that~$\Trace_\lambda(X_{\rho',E,1})$ is non-zero for a suitable~$\lambda$, since its image is invariant under any symmetric twist with an element of~$F_0^\times$, so must be trivial or~$X_{\rho,1,F}$. 

If~$E/F$ is of odd degree then the result follows immediately from Proposition~\ref{propBijectivityForEFodd}. Since the result is transitive in towers of extensions, this means we can reduce to the case that~$E/F$ is quadratic; in particular,~$E|F$ is at worst tamely ramified and we can take~$\lambda=\trace_{E|F}$. Moreover, we may replace~$E|F$ by~$E|F_0$ since, if~$\Trace_{E|F_0,\rho',1}(X_{\rho',1,E})$ is non-zero then~$\Trace_{E|F,\rho',1}(X_{\rho',1,E})$ is non-zero also. But then, by transitivity again and considering the extensions~$E|E_0$ and~$E_0|F_0$, we reduce to the case~$E|F$ quadratic with~$F=F_0$. Now Lemma~\ref{lemEFDeg2} implies that~$X_{\rho',1,E}$ is not in the kernel of~$\Trace_{E|F,\rho',1}$, as required.
\end{proof}


\section{Skolem--Noether}

In this section we consider Skolem--Noether-like theorems for classical groups.
We take the notation~$E,\rho',\lambda$ from Section~\ref{secWittGroups}.
We fix two~$\rho'$-$\sigma$-equivariant~$F$-algebra embeddings
\[
\phi_i: (E,\rho')\rightarrow (A,\sigma),\ i=1,2.
\]
We attach to each~$\phi_i$ an~$\epsilon$-hermitian form 
\[
h^{\phi_i}:V\times V\ra E
\] 
with respect to~$\rho'$ such that 
\[
h=\lambda\circ h^{\phi_i}.
\]
For the proof that such a form exists and is unique, see~\cite{broussousStevens:09}. Note that the~$\epsilon$-hermitian forms~$h^{\phi_i}$ differ because the maps~$\phi_i$ induce different~$E$-actions on~$V$. In particular, two such embeddings~$\phi_1,\phi_2$ are conjugate by an element of~$G$ if and only if~$(V,h^{\phi_1})$ is isomorphic to~$(V,h^{\phi_2})$ as an hermitian~$E$-space. 

We then get the following corollary of Theorem~\ref{thmMaximlaAnisotropicElement}.

\begin{corollary}\label{corSkolemNoether1}
Suppose that~$\rho'$ is non-trivial and that either~$\epsilon=1$ or~$F\neq F_0$. Then~$\phi_1,\phi_2$  are conjugate by an element of~$g\in G$, that is
\[
g\phi_1(x)g^{-1}=\phi_2(x), \text{for all }x\in E.
\]
\end{corollary}

\begin{proof}
We write~$W_{\rho',\epsilon}(E)^0$ for the set of classes of~$W_{\rho',\epsilon}(E)$ with even-dimensional anisotropic part. Then~$W_{\rho',\epsilon}(E)^0$ only consists of the trivial element and~$X_{\rho',\epsilon,E}$ so, by Theorem~\ref{thmMaximlaAnisotropicElement},
there is a map~$\lambda$ such that~$\Trace_\lambda$ is injective on~$W_{\rho',\epsilon}(E)^0$. Since~$\Trace_\lambda(\langle h^{\phi_i}\rangle)=\langle h\rangle$, we deduce that~$(V,h^{\phi_1})$ and~$(V,h^{\phi_2})$ are isomorphic as hermitian~$E$-spaces and the result follows.
\end{proof}

In the symplectic case, the analogous result is false without further hypotheses. The following theorem gives a sufficient additional condition which will be useful. 

\begin{theorem}\label{thmSkolemNoether}
For~$i=1,2$, let~$\Lambda^i$ be a self-dual lattice sequence in~$V$ normalized by~$\phi_i(E)^\times$. 
Let~$\beta$ be a non-zero skew-symmetric element generating~$E$ over~$F$ and write~$r_i:=1+\nu_{\Lambda^i}(\phi_i(\beta))$. Suppose that there is an element~$g$ of~$G$ such that 
\[
g^{-1}\left(\phi_1(\beta)+\mf{a}_{r_1}(\Lambda^1)\right)g\cap
\left(\phi_2(\beta)+\mf{a}_{r_2}(\Lambda^2)\right)\ne 
\emptyset.
\] 
Then~$\phi_1,\phi_2$  are conjugate by an element of~$G$.
\end{theorem}

In the language of strata below (Section~\ref{secSemisimpleStrata}), the hypotheses say that the pure skew strata~$[\Lambda^i,r_i+1,r_i,\phi_i(\beta)]$ intertwine. We will need the following lemma, where we recall that~$h_\gamma$ denotes the twist of~$h$ by a (skew-)symmetric element~$\gamma$ (see Definition~\ref{def:twist})

\begin{lemma}\label{lemLiftOfAClass}
Let~$\Lambda$ be a self-dual lattice sequence and let~$a_1,a_2$ be two non-zero symmetric or skew-symmetric elements of the normalizer of~$\Lambda$ such that~$a_1a_2^{-1}\in\tilde{U}^s(\Lambda)$, for some~$s>0$.
Then there is an~$F$-linear isometry from~$(V,h_{a_1})$ to~$(V,h_{a_2})$ in~$\tilde{U}^{s}(\Lambda)$.
\end{lemma}

\begin{proof}
We apply Proposition~\ref{propliftResidualIsometries} for~$f=\id_V$ to see that the spaces~$(V,h_{a_1})$ and~$(V,h_{a_2})$ are isometric. Now we apply Corollary~\ref{corCohomologyArgument}, with~$f=\id_V$ again, to finish the proof. 
\end{proof}

\begin{proof}[Proof of Theorem~\ref{thmSkolemNoether}]
By Corollary~\ref{corSkolemNoether1} we only need to treat the case that~$F=F_0$ and~$\epsilon=-1$. 
By hypothesis, there are elements~$g\in G$ and~$c_i\in\phi_i(\beta)+\mf{a}_{r_i,-}(\Lambda^i)$ such that~$gc_1g^{-1}=c_2$. Thus, by Lemma~\ref{lemLiftOfAClass}, we have isometries
\[
h_{\phi_1(\beta)}\cong h_{c_1}\cong h_{c_2}\cong h_{\phi_2(\beta)},
\]
where the middle isomorphism is given by~$g$. Let~$f$ be an isomorphism from~$h_{\phi_1(\beta)}$ to~$h_{\phi_2(\beta)}$. Since~$h_{\phi_i(\beta)}$ are orthogonal forms, Corollary~\ref{corSkolemNoether1} applied to the embeddings~$x\mapsto f\phi_1(x)f^{-1}$ and~$\phi_2$ implies that there is an isomorphism from~$h_{\phi_1(\beta)}$ to~$h_{\phi_2(\beta)}$ which conjugates~$\phi_1$ to~$\phi_2$. But any such isomorphism is an isometry of~$(V,h)$, as required. 
\end{proof}

We will also need the following integral version of the Skolem--Noether theorem:

\begin{proposition}[{{\cite[Theorem~1.2]{skodlerack:14}}}]\label{propConjug}
Let~$\phi_i: (E,\rho')\rightarrow (A,\sigma)$ be a~$\rho'$-$\sigma$-equivariant~$F$-algebra embedding, for~$i=1,2$.
Suppose further that~$(V,h^{\phi_1})$ is isomorphic to~$(V,h^{\phi_2})$ as hermitian
$E$-spaces and that there is a self-dual lattice chain~$\Lambda$ normalized by~$\phi_i(E)^\times$,~$i=1,2$. Then~$\phi_1,\phi_2$ are conjugate by an element of~$U(\Lambda)$.
\end{proposition}

\section{Semisimple strata}\label{secSemisimpleStrata}

We now turn to the notion of semisimple stratum for~$G$. The background can be found in~\cite{bushnellKutzko:93,stevens:02,stevens:05}, whose notation we adopt. However, many of the results in the literature are only available for lattice chains, while other results on semisimple strata were omitted in~\cite{stevens:05} (jumping directly to semisimple \emph{characters}). Thus we gather together here various results which we will need in our work.

A \emph{stratum} is a quadruple~$[\Lambda,q,r,\beta]$ consisting of an~$o_F$-lattice sequence~$\Lambda$, non-negative integers~$q\ge r$ and an element~$b\in\mf{a}_{-q}(\Lambda)$.
This stratum is called \emph{strict} if~$\Lambda$ is a lattice chain. 
The stratum is \emph{skew} if~$\beta\in A_-$ and~$\Lambda$ is self-dual, and it is called \emph{zero} if~$\beta=0$ and~$q=r$. 

Two strata~$[\Lambda,q,r,\beta]$ and~$[\Lambda',q',r',\beta']$ are \emph{equivalent} if 
$\beta+\mf{a}_{-r-j}=\beta'+\mf{a}'_{-r'-j}$,
for all non-negative integers~$j$. They \emph{intertwine} under a subgroup~$H$ of~$\tilde{G}$ if there is an element 
$g$ of~$H$ such that~$g(\beta+\mf{a}_{-r})g^{-1}$ intersects~$\beta'+\mf{a}'_{-r'}$. We denote the set of such elements
by~$I_H([\Lambda,q,r,\beta],[\Lambda',q',r',\beta'])$. If both strata are equal we skip the second argument and if~$H$ is~$\tilde{G}$ we skip~$H$ in the notation.
The two strata are conjugate under~$H$ if there is a~$g\in H$
such that~$[g\Lambda,q,r,g\beta g^{-1}]$ is equivalent to~$[\Lambda',q',r',\beta']$. 

\begin{definition}[Simple stratum]\label{defDegreeSimple}
A stratum~$[\Lambda,q,r,\beta]$ is called
\be
\item \emph{pure} if~$F[\beta]$ is a field such that~$F[\beta]^\times\subseteq\mf{n}(\Lambda)$ and~$\nu_\Lambda(\beta)=-q<-r$;
\item \emph{simple} if either it is zero, or it is pure and the degree~$[F[\beta]:F]$ is minimal among all equivalent pure strata.
\ee
\end{definition}

This is equivalent to~\cite[Definition~1.5]{stevens:05}, or~\cite{bushnellKutzko:93} in the case of lattice chains (see Proposition~\ref{propDegreeSimpleAndk0Simple} below). 

We now want to consider strata where~$F[\beta]$ is semisimple and for that we need to consider direct sums in~$V$. Given a decomposition~$V=\bigoplus_i V^i$ we write~$A^{i,j}$ for~$\Hom_F(V^j,V^i)$ and~$1^i$ for the projection onto~$V^i$ with kernel~$\bigoplus_{j\ne i}V^j$. A stratum~$[\Lambda,q,r,\beta]$ is \emph{split} by the decomposition if~$1^i\beta 1^j=0$ for~$i\ne j$ and if it splits~$\Lambda$, i.e.~$\Lambda$ is the direct sum of the lattice sequences~$\Lambda^i:=\Lambda\cap V^i$. We write~$\beta_i:=1^i\beta 1^i$ and~$q_i:=-\min\{\nu_\Lambda(\beta_i),-r\}$. We are now in a position to define a semisimple stratum.

\begin{definition}[{{\cite[Definition~3.2]{stevens:05}}}]
A stratum~$[\Lambda,q,r,\beta]$ is called \emph{semisimple} if either it is zero or~$\nu_\Lambda(\beta)=-q<-r$ and there is a splitting~$V=\bigoplus_iV^i$ such that 
\be
\item\label{def:ss.ii} for every~$i$ the stratum~$[\Lambda^i,q_i,r,\beta_i]$ in~$A^{i,i}$ is simple,
\item for all~$i\neq j$ the stratum~$[\Lambda^i\oplus\Lambda^j,\max\{q_i,q_j\},r,\beta_i+\beta_j]$ is not equivalent to any simple stratum.
\ee
A semisimple stratum is called \emph{skew-semisimple} if the decomposition of~$V$ is orthogonal and all strata occurring in~\ref{def:ss.ii} are skew. 
\end{definition}

For later, to describe the intertwining of~$[\Lambda,q,r,\beta]$, we need an integer~$k_0(\beta,\Lambda)$ which characterizes the semisimplicity of a stratum. Denote by~$a_\beta: A\ra A$ the map~$a_\beta(x)=x\beta-\beta x$ and put~$\mf{n}_l=a^{-1}_\beta(\mfa_l)\cap \mfa_0$. If~$F[\beta]$ is a field we define, as in Definition~\cite[1.4]{stevens:05}: 
\[
k_0(\beta,\Lambda):=\max\{-q,\max\{l\in \bbZ\mid \mf{n}_l\not\subseteq \mf{b}_0+\mfa_1\}\},
\]
and one writes~$k_F(\beta)$ for~$k_0(\beta,\mf{p}_E^\bbZ)$, where~$\mf{p}_E^\bbZ$ denotes the lattice sequence~$i\mapsto\mf{p}_E^i$, the unique~$o_F$-lattice chain in the~$F$-vector space~$E$ whose normalizer contains~$E^\times$. We have that 
\begin{equation}\label{eqk0kF}
 k_0(\beta,\Lambda)=e(\Lambda|o_E)k_F(\beta),
\end{equation}
by the remark after~\cite[Lemma~5.6]{stevens:01}. 
We now prove that Definition~\ref{defDegreeSimple} is equivalent to that in~\cite[1.5]{stevens:05} 

\begin{proposition}\label{propDegreeSimpleAndk0Simple}
Given a non-negative integer~$s$, a pure stratum~$[\Lambda,q,s,\beta]$ is simple if and only if~$-s<k_0(\beta,\Lambda)$. Further, writing~$\tilde{\Lambda}=\bigoplus_{l=0}^{e-1}(\Lambda-l)$, with~$e=e(\Lambda|o_F)$, we have~$k_0(\beta^{\oplus e},\tilde{\Lambda})=k_0(\beta,\Lambda)$.
\end{proposition}

Note that the lattice sequence~$\tilde{\Lambda}$ in the statement is in fact a lattice chain, with the same period as~$\Lambda$.

\begin{proof}
The second assertion follows directly from~\eqref{eqk0kF}, and we thus only concentrate on the first, which is true if~$\Lambda$ is a lattice chain by~\cite[Theorem~2.4.1]{bushnellKutzko:93}. 
We compare the two notions of simple: a stratum which is simple in the sense of Definition~\ref{defDegreeSimple} is called \emph{degree-simple},
and a stratum which is either zero or pure satisfying~$-s<k_0(\beta,\Lambda)$ is called \emph{$k_0$-simple}. 

If~$[\Lambda,q,s,\beta]$ is~$k_0$-simple then so is~$[\tilde{\Lambda},q,s,\beta^{\oplus e}]$, by the second assertion, and thus it is degree-simple, because~$\tilde{\Lambda}$ is a lattice chain. Thus~$[\Lambda,q,s,\beta]$ is degree-simple. 

If~$[\Lambda,q,s,\beta]$ is degree- but not~$k_0$-simple, then~$[\mf{p}_E^\bbZ,-\nu_E(\beta),\floor{\frac{s}{e(\Lambda|o_E)}} ,\beta]$ is not~$k_0$-simple. 
But then, the latter is not degree-simple, because~$\mf{p}_E^\bbZ$ is a lattice chain, and thus~$[\Lambda,q,s,\beta]$ is not degree-simple, using a~$(W,E)$-decomposition as in~\cite[Lemma~5.3]{bushnellKutzko:99}.
\end{proof} 

\begin{corollary}
$[\bigoplus_{l=0}^{e-1}(\Lambda-l),q,s,\beta^{\oplus e}]$ is simple if and only if~$[\Lambda,q,s,\beta]$ is simple.
\end{corollary}

If~$F[\beta]$ is not a field we define for a semisimple stratum, as in~\cite[(3.6)]{stevens:05},
\[
k_0(\beta,\Lambda):=-\min\{s\in\bbZ\mid [\Lambda,q,s,\beta]\text{ is not semisimple}\}.
\]
This integer is negative because~$[\Lambda,q,r,\beta]$ is semisimple and~$r\geq 0$.

\subsubsection*{Minimal strata}

We begin now with an analysis of semisimple strata of the form~$[\Lambda,q,q-1,\beta]$. For the simple case, we recall that an element~$\beta$ of an extension~$E|F$ is called \emph{minimal} if it satisfies the following two conditions:
\be
\item~$\Gcd{\nu_{E}(\beta)}{e(E|F)}=1$;
\item~$\beta^{e(E|F)}\pi^{\nu_E(\beta)}+\mf{p}_E$ generates the extension~$\kappa_E|\kappa_F$.
\ee
Then, by~\cite[1.4.13(ii),1.4.15]{bushnellKutzko:93}, a pure stratum~$[\Lambda,q,q-1,\beta]$ is simple if and only if~$\beta$ is minimal. By a slight abuse, we call a semisimple stratum of the form~$[\Lambda,q,q-1,\beta]$ a \emph{minimal semisimple stratum}.

For minimal semisimple strata, the \emph{characteristic polynomial} is very important for distinguishing the summands. For~$b$ an element of a finite dimensional semisimple algebra~$B$ over some field~$K$, we denote the reduced characteristic polynomial of~$b$ in~$B|K$, defined in~\cite[(9.20)]{reiner:03}, by~$\chi_{x,B|K}$, and the minimal polynomial by~$\mu_{x,B|K}$.

\begin{definition}
Let~$[\Lambda,q,q-1,\beta]$ be a stratum with~$\nu_\Lambda(\beta)=-q$ and set~$y_\beta:=\beta^\frac{e}{g}\pi^{\frac{q}{g}}$, where~$g=\Gcd{e}{q}$, with characteristic polynomial~$\Phi(X)=\chi_{y_\beta,A|F}\in o_F[X]$. We define the \emph{characteristic polynomial} of the stratum
$[\Lambda,q,q-1,\beta]$ to be the reduction~$\phi_\beta:=\bar{\Phi}\in\kappa_F[X]$. It depends only on the equivalence class of the stratum.

For a zero stratum we define~$y_0:=0$ and~$\phi_0(X):=X^{N}$, where~$N=\dim_F(V)$.
\end{definition}

\begin{remark}
If~$[\Lambda,q,q-1,\beta]$ and~$[\Lambda,q,q-1,\gamma]$ intertwine then~$\phi_{\beta}=\phi_{\gamma}$.
\end{remark}




\begin{proposition}\label{propSemisimplePolynomial}
If~$[\Lambda,q,q-1,\beta]$ is semisimple with associated splitting~$V=\bigoplus_{i\in I} V^i$, then we have the following:
\be
\item~$\phi_\beta$ is the product of the polynomials~$\phi_{\beta_i}$, which are pairwise coprime polynomials;
\item each polynomial~$\phi_{\beta_i}$ is a power of an irreducible polynomial;  
\item the~$F$-algebra homomorphism induced by~$\beta\mapsto \sum_{i\in I}\beta_i$ is a bijection from~$F[\beta]$ to the product of the~$E_i:=F[\beta_i]$;
\item~$\kappa_F[\bar{y}_\beta]$ is canonically isomorphic to~$\prod_{i\in I}\kappa_F[\bar{y}_{\beta_i}]$. 
\ee
\end{proposition}

\begin{proof}
For all indices~$i$, we have~$e=e(\Lambda|o_F)=e(\Lambda^i|o_F)$ and~$q=q_i$ for all indices~$i$ with~$\beta_i\neq 0$. Since also~$\beta=\sum_i\beta_i$ with~$\beta_i\in A^{i,i}$, we get
\[
y_\beta=\sum_i\beta_i^\frac{e}{g}\pi^{\frac{q}{g}}=\sum_iy_{\beta_i},
\]
and~$\phi_\beta$ is equal to the product of the~$\phi_{\beta_i}$. That~$\phi_{\beta_i}$ is primary now follows from the fact that~$[\Lambda^i,q_i,r,\beta_i]$ is a simple stratum and the remaining assertions are a consequence of~\cite[Remark~3.3]{stevens:05}.
\end{proof}

%
%
%

It will also be useful to have another criterion by which to recognise a minimal semisimple stratum. Recall that a stratum~$[\Lambda,q,q-1,\beta]$ is called \emph{fundamental} if the coset~$\beta+\mfa_{1-q}$ contains no nilpotent elements; in this case, the rational number~$\frac{q}{e}$ is called the \emph{level} of the stratum, where~$e=e(\Lambda|o_F)$. We also define the level of the zero stratum~$[\Lambda,q,q,0]$ to be~$\frac{q}{e}$.

\begin{proposition}\label{propFundamental}
A stratum~$[\Lambda,q,q-1,\beta]$ is fundamental if and only if its characteristic polynomial is not a power of~$X$. Two fundamental strata which intertwine have the same level. If a zero stratum intertwines a fundamental stratum, then they have different levels. 
\end{proposition}

\begin{proof}
Suppose~$[\Lambda,q,q-1,\beta]$ has characteristic polynomial~$X^m$ and put~$e=e(\Lambda|o_F)$; then the element~$\beta$ satisfies 
\[
\beta^{em}\in \pi^{-qm}\mfa_1=\mfa_{1-qme}.
\]
Then, by~\cite[Lemma~2.1]{bushnell:87}, there is a nilpotent element in~$\beta+\mfa_{1-q}$, so the stratum is not fundamental. (The proof of that Lemma is valid for lattice sequences if one allows block matrices with block sizes~$0\times l$ or~$l\times 0$.) Conversely, if~$[\Lambda,q,q-1,\beta]$ is not fundamental, then~$y_{\beta}$ is congruent to a nilpotent element modulo~$\mfa_1$, and thus the characteristic polynomial of the stratum is a power of~$X$. The remaining assertions now follow easily, because if one of them were false, then there would be a fundamental stratum whose characteristic polynomial is a power of~$X$.
\end{proof}

We now give criteria for a fundamental stratum to be simple or semisimple. We recall that a fundamental stratum is called \emph{non-split} if the characteristic polynomial of the stratum is a power of an irreducible polynomial. Given a fundamental stratum~$[\Lambda,q,q-1,b]$ we define the following~$\kappa_F$-algebra 
\[
\mathcal{R}([\Lambda,q,q-1,b]):=\{\bar{x}\in \mf{a}_0/\mf{a}_1\mid xb\equiv bx \pmod{\mf{a}_{1-q}}\}.
\]
The following result is stated in~\cite[2.4.13]{bushnellKutzko:93} for strict strata but, because the quotient~$\mf{a}_0/\mf{a}_1$ depends only on the image of~$\Lambda$, is also valid for arbitrary lattice sequences.
 
\begin{proposition}[{\cite[2.4.13]{bushnellKutzko:93}}]\label{propSimplCritForFundStratum}
A non-split fundamental stratum~$[\Lambda,q,q-1,b]$ is equivalent to a simple stratum if and only if~$\mathcal{R}([\Lambda,q,q-1,b])$ is semisimple.
\end{proposition}

To get a similar result for semisimple strata we need, for an element~$b\in \mf{a}_{-q}(\Lambda)$ and an integer~$n$, the map
\[
m_{n,q,b}: \mf{a}_{-nq}/\mf{a}_{1-nq}\ra\mf{a}_{-(n+1)q}/\mf{a}_{1-(n+1)q}
\]
induced by multiplication by~$b$.

\begin{proposition}\label{propSemiSimplCritForFundStratum}
A fundamental stratum~$[\Lambda,q,q-1,b]$ is equivalent to a semisimple stratum if and only if~$\mathcal{R}([\Lambda,q,q-1,b])$ is semisimple and, for all non-negative integers~$n$, the kernel of~$m_{n+1,q,b}$ and the image of~$m_{n,q,b}$ intersect trivially. 
\end{proposition}

\begin{proof}
Since the algebra~$\mathcal{R}([\Lambda,q,q-1,b])$ and the maps~$m_{n,q,b}$ depend only on the equivalence class of the stratum, we are free to move to an equivalent stratum at any point. 

Suppose first that~$\mathcal{R}([\Lambda,q,q-1,b])$ is semisimple and, for all non-negative integers~$n$, the kernel of~$m_{n+1,q,b}$ and the image of~$m_{n,q,b}$ intersect trivially. We inductively find a splitting. For this, assume that~$\phi_b$ is a product of two coprime monic factors~$\bar{f_0}$ and~$\bar{f_1}$. Let~$\Phi$ be the characteristic polynomial of~$y_b=\pi^{\frac{q}{g}}b^{\frac{e}{g}}$, where~$g$ is the greatest common divisor of~$e=e(\Lambda|o_F)$ and~$q$. Hensel's Lemma implies that we can factorize~$\Phi$ as~$f_0f_1$ where~$f_i$ is a monic lift of~$\bar{f_i}$. By B\'ezout's Lemma, there are polynomials~$a_0,a_1\in o_F[X]$ such that~$a_0f_0+a_1f_1=1$. The map~$1_i=a_i(y_b)f_i(y_b)$ is the projection onto the kernel of~$f_{1-i}$, and the sum~$\ker(f_0)\oplus \ker(f_1)=V$ splits the stratum~$[\Lambda,q,q-1,b]$. Moreover, we have~$\mathcal{R}([\Lambda,q,q-1,b])\simeq\mathcal{R}([\Lambda^0,q_0,q-1,b_0])\oplus\mathcal{R}([\Lambda^1,q_1,q-1,b_1])$, by the coprimality of~$f_0,f_1$, so that both~$\mathcal{R}([\Lambda^i,q_i,q-1,b_i])$ are semisimple.

Thus, by Proposition~\ref{propSimplCritForFundStratum} and~\ref{propSemisimplePolynomial}, we only have to show that strata equivalent to zero strata are the only non-fundamental strata for which the kernel of~$m_{n+1,q,b}$ and the image of~$m_{n,q,b}$ intersect trivially. Now let us assume that~$[\Lambda,q,q-1,b]$ is non-fundamental. Then without loss of generality we can assume that~$b$ is nilpotent. The conditions on the maps imply that~$m_{n,q,b}\circ m_{(n-1),q,b}\circ\cdots\circ m_{0,q,b}$ is injective on the image of~$m_{0,q,b}$. If~$n$ is big enough, the first product is the zero map, so the image of~$m_{0,q,b}$ is zero, i.e.~$[\Lambda,q,q-1,b]$ is equivalent to a zero stratum. 

For the converse, suppose that~$[\Lambda,q,q-1,b]$ is a semisimple stratum with associated splitting~$V=\bigoplus_{i\in I}V^i$. Since the characteristic polynomials~$\phi_{b_i}$ are pairwise coprime, we have~$\mathcal{R}([\Lambda,q,q-1,b])\simeq\bigoplus_{i\in I}\mathcal{R}([\Lambda^i,q_i,q-1,b_i])$ and, since each stratum~$[\Lambda^i,q_i,q-1,b_i])$ is simple, this algebra is semisimple by Proposition~\ref{propSimplCritForFundStratum}. (Note that the algebra is clearly semisimple for the zero simple stratum.) 

The maps~$m_{n,q,b}$ preserve the decomposition~$A=\bigoplus{i,j} A^{i,j}$ so we may work blockwise. On the diagonal blocks~$A^{i,i}$, the map~$m_{n,q,b}$ is either zero (in the case~$b_i=0$) or bijective. On the non-diagonal blocks~$A^{i,j}$, with~$i\ne j$, the map is bijective or zero by~\cite[3.7~Lemma~4]{bushnellKutzko:99}.
\end{proof}

\subsubsection*{Semisimple strata}

Now we turn to the case of general semisimple strata~$[\Lambda,q,r,\beta]$. A very important tool to prove properties of semisimple strata by an inductive procedure is the tame corestriction map, which was introduced in~\cite[1.3.3]{bushnellKutzko:93} in the simple case. 

\begin{definition}\label{def:tamecor}
Let~$E|F$ be a field extension and~$B$ be the centralizer of~$E$ in~$A$. A non-zero~$B$-$B$-bimodule map~$s:A\to B$ is called a \emph{tame corestriction (relative to~$E|F$)} if, for all~$\mf{o}_F$-lattice sequences~$\Lambda$ normalized by~$E^{\times}$, we have
\[
s(\mf{a}_j(\Lambda))=\mf{a}_j(\Lambda)\cap B,
\]
for all integers~$j$. 

If~$E=F[\gamma]$ we often write~$s_\gamma$ for a (choice of) tame corestriction relative to~$E|F$. 
\end{definition}

\begin{remark}\label{rmk:tamecor}
\begin{enumerate}
\item By~\cite[1.3.4]{bushnellKutzko:93}, tame corestrictions exist: if~$\psi_F$ and~$\psi_E$ are additive characters of~$F$ and~$E$ respectively then there is a unique map~$s:A\ra B$ such that 
\[
\psi_F\circ\tr_{A|F}(ab)=\psi_ E\circ \tr_{B|E}(s(a)b),\qquad a\in A, \ b\in B.
\]
This map is a tame corestriction and every tame corestriction arises in this way. Moreover, tame corestrictions are unique up to multiplication by an element of~$o_E^\times$.
\item If~$\gamma$ generates the extension~$E|F$ then, by~\cite[1.3.2 (i)]{bushnellKutzko:93}, the kernel of~$s_\gamma$ is equal to the image of the adjoint map~$a_\gamma: A\ra A$.
\item If~$E$ is~$\sigma$-invariant, we can arrange the additive characters~$\psi_F$ and~$\psi_E$ in~(i) to be~$\sigma$-invariant also, and then the tame corestriction~$s$ is~$\sigma$-equivariant.
\end{enumerate}
\end{remark}

Given a simple stratum~$[\Lambda,q,r+1,\gamma]$ in~$A$ and element~$c\in\mfa_{-r}$, the tame corestriction map allows us to define a \emph{derived stratum}~$[\Lambda,r+1,r,s_\gamma(c)]$ in~$B_\gamma$, the centralizer in~$A$ of~$\gamma$, and we can ask whether this derived stratum is (equivalent to) a fundamental or simple stratum. The following theorem is particularly useful.

\begin{theorem}[{\cite[Theorems~2.2.8,~2.4.1]{bushnellKutzko:93}}]\label{thmSimplederived}
Let~$[\Lambda,q,r+1,\beta]$ be a stratum equivalent to a simple stratum~$[\Lambda,q,r+1,\gamma]$.
Then~$[\Lambda,q,r,\beta]$ is equivalent to a simple stratum if and only if the derived stratum~$[\Lambda,r+1,r,s_\gamma(\gamma-\beta)]$ is equivalent to a simple stratum.
\end{theorem}

As an immediate corollary, we get the following result on semisimple strata.

\begin{corollary}\label{corDerivedSplitting}
Let~$[\Lambda,q,r+1,\beta]$ be a stratum equivalent to a simple stratum~$[\Lambda,q,r+1,\gamma]$. Assume that we have a decomposition~$V=\bigoplus_i V^i$ into~$\beta$- and~$\gamma$-invariant~$F$-subspaces. Then~$[\Lambda,r+1,r,s_\gamma(\gamma-\beta)]$ is equivalent to a semisimple stratum with associated splitting~$V=\bigoplus_i V^i$ if and only if~$[\Lambda,q,r,\beta]$ is equivalent to a semisimple stratum with associated splitting~$V=\bigoplus_i V^i$.
\end{corollary}

\medskip

Suppose now that~$[\Lambda,q,0,\beta]$ is semisimple so that, for any~$0\le r< q$, the stratum~$[\Lambda,q,r+1,\beta]$ is equivalent to a semisimple stratum~$[\Lambda,q,r+1,\gamma]$. Then we can realize the assumption on~$\gamma$ in the previous corollary (that is, we can find~$\gamma$ such that the splitting associated to~$[\Lambda,q,0,\beta]$ is preserved by~$\gamma$) by the following theorem.

\begin{theorem}[{\cite[3.4]{stevens:05}}]\label{thmDiagonalizedSimple}
Let~$[\Lambda,q,r,\beta]$ be a (skew)-stratum split by~$V=\bigoplus_iV^i$ such that every stratum~$[\Lambda,q_i,r,\beta_i]$ is pure, and such that~$[\Lambda,q,r+1,\beta]$ is equivalent to a simple stratum. Then~$[\Lambda,q,r+1,\beta]$ is equivalent to a (skew)-simple stratum~$[\Lambda,q,r+1,\gamma]$ split by the same direct sum.
\end{theorem}

In particular, if~$[\Lambda,q,r,\beta]$ is a semisimple stratum with splitting~$V=\bigoplus_iV^i$ and~$[\Lambda,q,r+1,\beta]$ is equivalent to a simple stratum~$[\Lambda,q,r+1,\gamma]$ such that~$\gamma V^i\subseteq V^i$ for each~$i$, then Corollary~\ref{corDerivedSplitting} implies that the derived stratum~$[\Lambda,r+1,r,s_\gamma(\gamma-\beta)]$ is equivalent to a semisimple stratum with the same splitting~$V=\bigoplus_iV^i$.

\begin{notation}
For the rest of the article we use the following notation:~$[\Lambda,q,r,\beta]$ always denotes a stratum, and~$B$ the centralizer of~$\beta$ in~$A$. If~$[\Lambda,q,r,\beta]$ is semisimple then~$V=\bigoplus_{i\in I}V^i$ is the associated splitting and we have~$A=\bigoplus_{i,j} A^{i,j}$ and~$B=\bigoplus_{i\in I} B^{i,i}$, where~$B^{i,i}$ is the centralizer of~$E_i=F[\beta_i]$ in~$A^{i,i}$. Further, we write~$\mf{b}_l$ for the intersection of~$\mfa_l$ with~$B$. We use analogous notations for a second stratum~$[\Lambda',q',r',\beta']$ but all with~$()'$. If we want to specify the centralizer of~$\gamma$ in~$A$, for an arbitrary element~$\gamma$, we write~$B_\gamma$. 
\end{notation}

Let~$[\Lambda,q,r,\beta]$ be a semisimple stratum. We define a \emph{tame corestriction}~$s_\beta: A\rightarrow B$ for~$\beta$ by~$s_\beta(a):=\sum_{i} s_i(a_{ii})$, where 
$s_i$ is a tame corestriction for~$\beta_i$ as in Definition~\ref{def:tamecor}. If~$s_i$ is defined relative to additive characters~$\psi_F,\psi_{E^i}$ as in Remark~\ref{rmk:tamecor}(i), then we put~$\psi_{B^{i,i}}=\psi_{E^i}\circ\tr_{B^{i,i}|E^i}$ and define an additive character of~$B$ by
\[
\psi_B(b)=\prod_{i\in I}\psi_{B^{i,i}}(b_i),\qquad b=\sum_{i\in I}b_i,\quad b_i\in B^{i,i}.
\] 
Writing~$\psi_A=\psi_F\circ\tr_{A|F}$, the map~$s_\beta$ is then a non-zero~$(B,B)$-bimodule homomorphism satisfying 
\[
\psi_A(ab)=\psi_B(s_\beta(a)b), \qquad a\in A,\ b\in B,
\]
and
\[
 s_\beta(\mfa'_l)=\mf{b}'_l
\] 
for all lattice sequences~$\Lambda'$ which are split by~$V=\bigoplus_iV^i$ into a direct sum of~$o_{E^i}$-lattice sequences. 

\begin{lemma}\label{lemKernels}
The sequence~$A\stackrel{a_\beta}{\rightarrow}A\stackrel{s_\beta}{\rightarrow}B$ is exact and the kernel of~$s_\beta$ is split by the decomposition~$A=\bigoplus A^{i,j}$. 
\end{lemma}

\begin{proof}
By definition, the kernel of~$s_\beta$ is the direct sum of the~$A^{i,j}$, for~$i\neq j$, and of the kernels of~$s_i$, for~$i\in I$. 
The sequence is exact on the~$(i,i)$ components, by~\cite[1.3.2]{bushnellKutzko:93}, and it is therefore enough to prove that for~$j\neq i$ the restriction of~$a_\beta$ on~$A^{i,j}$ is bijective onto~$A^{i,j}$. It has the form~$a_\beta(a_{ij})=\beta_i a_{ij}-a_{ij}\beta_j$, which is injective because~$\beta_i$ and~$\beta_j$ have no common eigenvalue, because their minimal polynomials are coprime since~$[\Lambda^i\oplus \Lambda^j,\max\{q_i,q_j\},r,\beta_i+\beta_j]$ is not equivalent to a simple stratum.
\end{proof}

To describe the intertwining of a semisimple stratum~$[\Lambda,q,r,\beta]$, recall that we have defined the integer~$k_0=k_0(\beta,\Lambda)$ and the lattices~$\mf{n}_l=a_\beta^{-1}(\mfa_l)\cap\mfa_0$, for~$l$ an integer. We will also need the unit subgroups~$1+\mf{m}_l$, where
$\mf{m}_l=\mf{n}_{l+k_0}\cap \mfa_l$, for integers~$l\geq 1$. As the first of several intertwining results we have:

\begin{theorem}[{see~\cite[4.4]{stevens:05},~\cite[1.5.8]{bushnellKutzko:93} for simple strata}]\label{thmIntertwiningOneSemisimpleStratum} Let~$[\Lambda,q,r,\beta]$ be a semisimple stratum.
\begin{enumerate}
 \item~$I([\Lambda,q,r,\beta])=(1+\mf{m}_{-(k_0+r)})B^\times (1+\mf{m}_{-(k_0+r)}).$
 \item If the stratum is skew then
 \[
 I_G([\Lambda,q,r,\beta])=(G\cap (1+\mf{m}_{-(k_0+r)}))(G\cap B^\times)(G\cap (1+\mf{m}_{-(k_0+r)})).
 \]
\end{enumerate}
\end{theorem}

The crucial ingredient for the proof is:

\begin{lemma}[{\cite[3.7]{stevens:05}}]\label{lemDescrmfn}
For all integers~$s$ we have 
\begin{itemize}
 \item[(i)]~$\mf{n}_{-s}^{i,j}\subseteq \mfa_{-(k_0+s)}$ for~$i\neq j$.
 \item[(ii)]~$\mf{n}_{-s}=\mf{b}_0+\mf{n}_{-s}\cap \mfa_{-(k_0+s)}$
\end{itemize}
\end{lemma}

In~\cite{stevens:05}, this lemma was formulated for~$s\leq -k_0$, but the case~$s\geq -k_0$ is trivial. From this we deduce 

\begin{lemma}\label{lemabetaHomoeforineqj}
Take~$i\neq j$. The restriction of~$a_\beta$ to~$A^{i,j}$ is an~$F$-linear homeomorphism and~$a^{-1}_\beta(\mfa_s)^{i,j}$ is equal to 
$\mf{n}_s^{i,j}$ for all integers~$s\geq k_0$.
\end{lemma}

\begin{proof}
The map~$a_\beta$ is a linear automorphism on~$A^{i,j}$ by Lemma~\ref{lemKernels}; thus the image of an~$o_F$-lattice contains an~$o_F$-lattice and the restriction of~$a_\beta$ to~$A^{i,j}$ is a homeomorphism. It follows that, for~$s$ big enough, we have that~$a^{-1}_\beta(\mfa_s)^{i,j}:=a^{-1}_\beta(\mfa_s)\cap A^{i,j}$ is contained in~$\mfa_0^{i,j}$ and is therefore equal to~$\mf{n}_s^{i,j}$; in particular, it is contained in~$\mfa_{-k_0+s}^{i,j}$ by Lemma~\ref{lemDescrmfn}(i). By periodicity we have that~$a^{-1}_\beta(\mfa_s)^{i,j}$ is contained in~$\mfa_{-k_0+s}^{i,j}$ for all integers~$s$ and thus~$\mf{n}_s^{i,j}$ is equal to~$a^{-1}_\beta(\mfa_s)^{i,j}$ for all integers~$s$ with~$s\geq k_0$. 
\end{proof}

\begin{proof}[Proof of Theorem~\ref{thmIntertwiningOneSemisimpleStratum}]
We follow the proof of~\cite[1.5.8]{bushnellKutzko:93}. For a zero stratum is nothing to prove, so we assume the stratum is non-zero. The main ingredients which have to be verified are the exact sequences of~\cite[1.4.10]{bushnellKutzko:93}, which hold by Lemma~\ref{lemDescrmfn}, and the analogue of~\cite[1.4.16]{bushnellKutzko:93}, which we prove now. We write~$d$ for~$-(r+k_0)$ and put
\begin{itemize}
 \item~$M=\mf{n}_{t+jd+k_0}\cap \mfa_{t+jd}\cap (\gamma_1 (\mf{n}_{-r}\cap\mfa_d)+(\mf{n}_{-r}\cap\mfa_d)\gamma_2 + (\mf{n}_{t+jd-r}\cap\mfa_{t+(j+1)d}))$,
 \item~$L=\mfa_{t+jd+k_0}\cap (\gamma_1\mfa_{-r}+\mfa_{-r}\gamma_2+\mfa_{t+jd-r})$,
\end{itemize}
for integers~$t\geq 0,\ j\geq 1$ and elements~$\gamma_1,\gamma_2$ of~$B^\times$. 
The sequence~$M\stackrel{a_\beta}{\ra}L\stackrel{s_\beta}{\ra} B$ is exact if all its restrictions on the~$A^{i,j}$ are exact. For~$i=j$ the proof is done in~\cite[(5.2)]{stevens:01} and for~$i\neq j$ it follows from Lemma~\ref{lemDescrmfn}(i),~\ref{lemabetaHomoeforineqj} and~\ref{lemKernels}.
The same cohomology argument as in~\cite[Corollary~4.14]{stevens:05} proves~(ii).
\end{proof}

A completely analogous proof using~\cite[1.5.12]{bushnellKutzko:93} provides:

\begin{theorem}\label{thmIntertwiningTwoSemisimpleStratumSameBeta}
Let~$[\Lambda,q,r,\beta]$ and~$[\Lambda',q',r',\beta]$ be semisimple strata in~$A$. Then 
\begin{enumerate}
 \item~$I([\Lambda,q,r,\beta],[\Lambda',q',r',\beta])=(1+\mf{m}'_{-(k'_0+r')})B^\times (1+\mf{m}_{-(k_0+r)}).$
 \item If both strata are skew then
 \[
 I_G([\Lambda,q,r,\beta],[\Lambda',q',r',\beta])=(G\cap (1+\mf{m}'_{-(k'_0+r')}))(G\cap B^\times)(G\cap (1+\mf{m}_{-(k_0+r)})).
 \]
\end{enumerate}
\end{theorem}


\section{Matching for intertwining strata}

In this chapter we show that, if we have semisimple strata which intertwine, then there is a canonical bijection between their associated splittings. This will then allow us to deduce a Skolem--Noether theorem for skew-semisimple strata which intertwine.

\subsection{For general linear groups}
We fix a pair of semisimple strata~$[\Lambda,q,r,\beta]$ and~$[\Lambda',q',r',\beta']$, with splittings~$\bigoplus_{i\in I}V^i$ and~$\bigoplus_{j\in I'}V'^j$ respectively. The main result of this subsection is:

\begin{proposition}\label{propMatching}
Suppose that~$[\Lambda,q,r,\beta]$ intertwines~$[\Lambda',q',r,\beta']$ and that~$\Lambda,\Lambda'$ have the same period. Then 
\be
\item If one stratum is zero and~$q=q'$ then the other stratum is zero. 
\item If both strata are non-zero then~$q=q'$ and there is a unique bijection~$\zeta:I\to I'$ such that~$[\Lambda^i\oplus \Lambda'^{\zeta(i)},\max\{q_i,q'_{\zeta(i)}\},r,\beta_i+\beta'_{\zeta(i)}]$ is equivalent to a simple stratum, for all indices~$i\in I$. Moreover,~$V^i$ and~$V'^{\zeta(i)}$ have the same~$F$-dimension. 
\ee
\end{proposition}

Note that, in case~(i), both splittings are trivial so we trivially have a (unique) bijection~$\zeta$ as in~(ii). We call the bijection~$\zeta$ a \emph{matching} of~$[\Lambda',q',r,\beta']$ and~$[\Lambda,q,r,\beta]$. 

\begin{remark}
If~$\Lambda,\Lambda'$ do not have the same period then we can scale them so that they do. In particular, we only require the intertwining hypothesis in Proposition~\ref{propMatching} in order to get a matching~$\zeta$.
\end{remark}

%

To prove Proposition~\ref{propMatching} (and, later, other results on semisimple strata), we introduce the notion of a defining sequence for a semisimple stratum, which allows us to prove properties of semisimple strata by an inductive process (cf.~\cite{bushnellKutzko:93} for the simple case). 
\begin{itemize}
 \item A defining sequence for a minimal semisimple stratum is the stratum itself.
 \item Let~$[\Lambda,q,r,\beta]$ be a (skew)-semisimple stratum with~$r<q-1$ and associated splitting~$V=\bigoplus_{i\in I}V^i$. Then there are a partition~$I=\bigcup I_j$ and an element~$\gamma\in \prod_i A^{i,i}$ such that~$[\Lambda,q,r+1,\gamma]$ is a (skew)-semisimple stratum equivalent to~$[\Lambda,q,r+1,\beta]$ with associated splitting~$V=\bigoplus_{j}V^{I_j}$ (see Theorem~\ref{thmDiagonalizedSimple}). A defining sequence for~$[\Lambda,q,r,\beta]$ is the stratum itself together with a defining sequence for~$[\Lambda,q,r+1,\gamma]$. We call~$[\Lambda,q,r+1,\gamma]$ the \emph{first element} in the chosen defining sequence. 
\end{itemize}
Note that there is a significant degree of choice in producing a defining sequence for a semisimple stratum.

Suppose now we want to prove a statement~$P(\Delta,\Delta')$ for all pairs of semisimple strata~$\Delta,\Delta'$. The inductive procedure, which we call \emph{strata induction}, to prove~$P$ is given by the following steps. 
\begin{itemize}
 \item The base case: Here one proves~$P$ for all minimal semisimple strata. 
 \item The induction step: 
 \begin{enumerate}
  \item The step~$r+1$ to~$r$: From the induction hypothesis and possibly an auxiliary statement~{\bf (S1)} we restrict to the case where the first elements~$\Delta^{(1)}$ and~$\Delta'^{(1)}$ of defining sequences of~$\Delta$ and~$\Delta'$ have the same element~$\gamma$, and hence the same associated splitting. 
  \item Taking a second auxiliary statement~{\bf (S2)}, we show that the derived strata~$s_\gamma(\Delta)$ and~$s_\gamma(\Delta')$ satisfy the assumptions of~$P$. In this article,~(S2) will always be the description of the intertwining of~$\Delta^{(1)}$ with~$\Delta'^{(1)}$ (see Proposition~\ref{propS2} below). 
  \item The base case shows~$P(s_\gamma(\Delta),s_\gamma(\Delta'))$ and, together with a third auxiliary statement~{\bf (S3)}, provides~$P(\Delta,\Delta')$. For~(S3) we will use Theorem~\ref{thmSimplederived}. 
 \end{enumerate}
Strata induction can be restricted to simple strata by substituting the word semisimple by simple.
\end{itemize}

In the following, we use the notation for tame corestrictions as in the previous section.

\begin{proposition}\label{propS2}
Let~$[\Lambda,q,r,\beta]$ and~$[\Lambda',q',r',\beta]$ be semisimple strata with splitting~$V=\bigoplus_{i\in I}V^i$. Suppose we are given elements~$a\in \mfa_{-r}$ and
$a'\in \mfa'_{-r'}$ such that there is an element~$g$ of~$\tilde{G}$ which intertwines~$[\Lambda,q,r-1,\beta+a]$ with~$[\Lambda',q',r'-1,\beta+a']$. Using Theorem~\ref{thmIntertwiningTwoSemisimpleStratumSameBeta}, write~$g=(1+u')b(1+v)$, with~$b\in B^\times$. Then the component~$b^{i,i}$ intertwines~$[\Lambda,r,r-1,s_i(a^{i,i})]$ with~$[\Lambda',r',r'-1,s_i(a'^{i,i})]$, for all~$i\in I$. 
\end{proposition}

\begin{proof}
Again we only  have to consider a non-zero element~$\beta$. This is essentially the calculation in~\cite[2.6.1]{bushnellKutzko:93} which we want to recall, to show that its validity for different semisimple strata. Note that the hypotheses imply that~$g$ certainly intertwines~$[\Lambda,q,r,\beta]$ and~$[\Lambda',q',r',\beta]$ so that, by Theorem~\ref{thmIntertwiningTwoSemisimpleStratumSameBeta}, we can write~$g=(1+u')b(1+v)$, with~$u'\in \mf{m}'_{-(k'_0+r')}$, $b\in B^\times$ and~$v\in\mf{m}_{-(k_0+r)}$.

Let~$(1+w')$ be the inverse of~$(1+u')$. By the intertwining property of~$g$, we have
\[
g(\beta+a)\equiv (\beta+a')g \pmod{g\mfa_{1-r}+\mfa'_{1-r'}g}.
\] 
Multiplying by~$(1+w')$ on the left and $(1+v)^{-1}$ on the right we obtain
\beq\label{eqstartIntertwining}
b(1+v)(\beta+a)(1+v)^{-1}\equiv (1+w')(\beta+a')(1+w')^{-1}b \pmod{b\mfa_{1-r}+\mfa'_{1-r'}b}.
\eeq
We firstly consider the right hand side. 
\begin{eqnarray*}
(1+w')(\beta+a')&=& \beta-a_\beta(w')+a'+w'a'+\beta w'\\
&=& (\beta-a_\beta(w')+a')(1+w') + w'a' - a'w' + a_\beta(w')w'\\
&\equiv & (\beta-a_\beta(w')+a')(1+w') \pmod{\mfa'_{1-r'}},
\end{eqnarray*}
because~$a'\in \mfa'_{-r'},\ w'\in\mfa'_1$ and~$a_\beta(w')\in \mfa'_{-r'}$.
A similar calculation for the left hand side and equation (\ref{eqstartIntertwining}) leads to 
\beq\label{eqMiddleIntertwining}
b(\beta-a_\beta(v)+a)\equiv (\beta-a_\beta(w')+a')b\pmod{b\mfa_{1-r}+\mfa'_{1-r'}b}.
\eeq
We apply~$s_\beta$ to get
\[
bs_\beta(a)\equiv s_\beta(a')b \pmod{b\mf{b}_{1-r}+\mf{b}'_{1-r'}b}
\] 
and thus~$b^{i,i}$ intertwines the derived strata~$[\Lambda,r,r-1,s_i(a^{i,i})]$ and~$[\Lambda',r',r'-1,s_i(a'^{i,i})]$, for all~$i\in I$.
\end{proof}

\begin{proposition}[{cf.~\cite[2.2.1]{bushnellKutzko:93}}]\label{propS3}
Suppose that~$[\Lambda,q,r,\beta]$ and~$[\Lambda',q',r',\beta]$ are simple and that there are~$a\in\mfa_{-r}$ and~$a'\in\mfa'_{-r'}$ such that~$s_\beta(a)\equiv s_\beta(a') \pmod{\mf{b}_{1-r}+\mf{b}'_{1-r'}}$. Then, there are elements~$w'\in \mf{m}'_{-(k'_0+r')}$ and~$v\in \mf{m}_{-(k_0+r)}$ such that
\[
(1+w')(\beta+a')(1+w')^{-1}\equiv (1+v)(\beta+a)(1+v)^{-1}\pmod{\mfa_{1-r}+\mfa'_{1-r'}}.
\]
Moreover, if the strata and the elements~$a$ and~$a'$ are skew and the strata intertwine in~$G$, then we can choose~$1+v$ and~$1+w'$ in~$G$. 
\end{proposition}

Before the proof let us recall that the \emph{Cayley transform} of an element~$v$ of~$(\mfa_1)_-$ is the element~$\left(1+\frac{v}{2}\right)\left(1-\frac{v}{2}\right)^{-1}$. It is an element of~$U^1(\Lambda)\subset G$.  

\begin{proof}
Let~$\beta$ be non-zero and write~$C$ for the kernel of~$s_\beta$. Without loss of generality we can assume that~$s_\beta(a)$ and~$s_\beta(a')$ are equal, since the map~$s_\beta:\mfa_{1-r}\to\mf{b}_{1-r}$ is surjective. The map 
\[
\mf{m}_{-r-k_0}+\mf{m}'_{-r'-k'_0}\stackrel{a_\beta}{\ra} C\cap (\mfa_{-r}+\mfa'_{-r'})
\] 
is surjective because~$C\cap (\mfa_{-r}+\mfa'_{-r'})$ is equal to~$(C\cap \mfa_{-r})+(C\cap\mfa'_{-r'})$ by~\cite[1.3.17]{bushnellKutzko:93} and~$a_\beta^{-1}(\mfa_{-r})$ is a subset of~$B+\mf{m}_{-r-k_0}$ by Lemma~\ref{lemDescrmfn}(ii). Thus, we can find~$w'\in\mf{m}_{-r'-k'_0}$ and~$v\in\mf{m}_{-r-k_0}$ to satisfy~\eqref{eqMiddleIntertwining} for~$b=1$. We now follow the calculation in the proof of Proposition~\ref{propS2} backwards to show the desired congruence. In the skew situation we can find skew-symmetric elements, say~$\tilde{v}$ and~$\tilde{w'}$, which satisfy~\eqref{eqMiddleIntertwining} and we define~$1+v$ and~$1+w'$ to be the Cayley transforms of~$\tilde{v}$ and~$\tilde{w'}$ respectively. 
\end{proof}

We need one final lemma before we can prove Proposition~\ref{propMatching}, which will play the role of~(S1).

\begin{lemma}\label{lemConjblocksSemisimple}
Let~$[\Lambda,q,r,\beta]$ and~$[\Lambda',q,r,\beta']$ be semisimple strata for which there is a (unique) bijection~$\zeta:I\to I'$ such that~$[\Lambda^i\oplus \Lambda'^{\zeta(i)},\max\{q_i,q'_{\zeta(i)}\},r,\beta_i+\beta'_{\zeta(i)}]$ is equivalent to a simple stratum, for all indices~$i\in I$, and~$\dim_F V^i=\dim_F V'^{\zeta(i)}$. Then there are an element~$g$ of~$\tilde{G}$ and an element~$\gamma\in \prod_i A^{i,i}$ such that~$V^i=gV'^{\zeta(i)}$, for all indices~$i\in I$, and:
\begin{itemize}
 \item~$[g\Lambda',q,r,g\beta'g^{-1}]$ is equivalent to~$[g\Lambda',q,r,\gamma]$;
 \item~$[\Lambda,q,r,\beta]$ is equivalent to~$[\Lambda,q,r,\gamma]$; and
 \item~$[g\Lambda',q,r,\gamma]$ and~$[\Lambda,q,r,\gamma]$ are semisimple strata with the same associated splitting~$V=\bigoplus_{i\in I}V^i$.
\end{itemize}
\end{lemma}

\begin{proof}
Applying Theorem~\ref{thmDiagonalizedSimple} to the strata~$[\Lambda^i\oplus \Lambda'^{\zeta(i)},\max\{q_i,q_{\zeta(i)}\},r,\beta_i+\beta'_{\zeta(i)}]$, for each~$i$, we find an equivalent simple stratum~$[\Lambda^i\oplus \Lambda'^{\zeta(i)},\max\{q_i,q_{\zeta(i)}\},r,\gamma_i+\gamma'_{\zeta(i)}]$; in particular,~$\gamma_i$ and~$\gamma'_{\zeta(i)}$ have the same irreducible minimal polynomial, and the same characteristic polynomial since~$\dim_F V^i=\dim_F V'^{\zeta(i)}$. 

Further, for~$i\ne j$ in~$I$, the stratum~$[\Lambda^{i}\oplus \Lambda^{j}, \max\{q_{i},q_{j}\},r,\gamma_{i}+\gamma_{j}]$ is equivalent to~$[\Lambda^{i}\oplus \Lambda^{j}, \max\{q_{i},q_{j}\},r,\beta_{i}+\beta_{j}]$, which is not equivalent to a simple stratum, so that the stratum~$[\Lambda,q,r,\gamma]$ is semisimple, where~$\gamma=\sum_{i\in I}\gamma_i$. The same applies to~$[\Lambda',q,r,\gamma']$, where~$\gamma'=\sum_{i\in I}\gamma'_{\zeta{i}}$. Finally, since~$\gamma_i$ and~$\gamma'_{\zeta(i)}$ have the same characteristic polynomial, we can find~$g\in G$ such that~$V^i=gV'^{\zeta(i)}$ and~$g\gamma' g^{-1}=\gamma$, and the result follows.
\end{proof}

\begin{proof}[Proof of Proposition~\ref{propMatching}]
(i) and the equality~$q=q'$ in~(ii) follow from the results on level in Proposition~\ref{propFundamental}. The existence of~$\zeta$ in~(ii) is proved by strata induction, where we take Lemma~\ref{lemConjblocksSemisimple} for~(S1), Proposition~\ref{propS2} for~(S2), and Theorem~\ref{thmSimplederived} for~(S3). The base case follows because the characteristic polynomials are equal, so we match the primary factors using~\cite[3.3(ii)]{stevens:05}. The equality of dimensions follows from the fact that the degree of the~$i$th primary factor is the dimension of~$V^i$. 

For the inductive step, suppose that~$\Delta=[\Lambda,q,r,\beta]$ and~$\Delta'=[\Lambda',q,r,\beta']$ are semisimple strata as in the proposition which intertwine. Then the stratum~$[\Lambda,q,r+1,\beta]$ is equivalent to a semisimple stratum~$\Delta_\gamma=[\Lambda,q,r+1,\gamma]$ whose splitting is a coarsening of that of~$\Delta$, by Theorem~\ref{thmDiagonalizedSimple}; similarly we have a semisimple stratum~$\Delta'_\gamma=[\Lambda',q,r+1,\gamma']$. Since the strata~$\Delta_\gamma$,~$\Delta'_\gamma$ intertwine, we may apply the inductive hypothesis to them. In particular, they satisfy the hypotheses of Lemma~\ref{lemConjblocksSemisimple} and, replacing~$\Delta'$ by its conjugate~$g\Delta'$, we may assume~$\gamma'=\gamma$.

Now we apply~(S2) -- Proposition~\ref{propS2} -- to the strata~$\Delta_\gamma$ and~$\Delta'_\gamma$, with~$a=\beta-\gamma$ and~$a'=\beta'-\gamma$. The conclusion is that the derived strata intertwine so that the base case gives us a bijection between the index sets~$\zeta: I\ra I'$ such that, for each~$i\in I$, the stratum~$[\Lambda^i\oplus \Lambda'^{\zeta(i)},r+1,r,s_{\gamma_i}(\beta_i-\gamma_i)+s_{\gamma_{\zeta(i)}}(\beta'_{\zeta(i)}-\gamma_{\zeta(i)})]$ is equivalent to a simple stratum. (Here~$\gamma_i=1^i\gamma 1^i$, where~$1^i$ is the idempotent corresponding to~$V^i$, and similarly~$\gamma'_{\zeta(i)}=1'^{\zeta(i)}\gamma 1'^{\zeta(i)}$, corresponding to~$V'^{\zeta(i)}$; note also that~$\gamma_i$ and~$\gamma_{\zeta(i)}$ have the same characteristic polynomial so that we can view both~$V^i$ and~$V'^{\zeta(i)}$ and~$F[\gamma_i]$-vector spaces.) But then~$[\Lambda^i\oplus \Lambda'^{\zeta(i)},\max\{q_i,q'_{\zeta(i)}\},r,\beta_i+\beta'_{\zeta(i)}]$ is equivalent to a simple stratum, by~(S3) -- Theorem~\ref{thmSimplederived}.

The existence of~$\zeta$ implies, in particular, that both strata have the same number of blocks, i.e. the sets~$I$ and~$I'$ have the same cardinality. Finally, we prove the uniqueness of~$\zeta$. Assume, for contradiction, that there are two distinct indices~$i,j\in I$ and an index~$i'\in I'$ such that~$[\Lambda^{i}\oplus \Lambda'^{i'},\max\{q_{i},q'_{i'}\},r,\beta_{i}+\beta'_{i'}]$ and~$[\Lambda^{j}\oplus \Lambda'^{i'},\max\{q_{j},q'_{i'}\},r,\beta_{j}+\beta'_{i'}]$ are both equivalent to simple strata. From this (and the equality of periods) it follows that the integers~$q_{i}$,~$q_{j}$ and~$q'_{i'}$ are all equal; we denote this integer by~$q$. 

By the proof of the existence, the spaces~$V^{i}$ and~$V^{i'}$ have the same dimension, and thus, by conjugating, we can assume that they are the equal. Then, by Theorem~\ref{thmDiagonalizedSimple}, the strata~$[\Lambda^{i},q,r,\beta_{i}]$ and~$[\Lambda'^{i'},q,r,\beta_{i'}]$ intertwine. Then the stratum~$[\Lambda^{i}\oplus\Lambda^{j},q,r,\beta_{i}+\beta_{j}]$ intertwines with~$[\Lambda'^{i'}\oplus\Lambda^{j},q,r,\beta_{i'}+\beta_{j}]$ and the latter is equivalent to a simple stratum. Thus the semisimple stratum~$[\Lambda^{i}\oplus\Lambda^{j},q,r,\beta_{i}+\beta_{j}]$, which has two blocks, is intertwined with a simple stratum, which has only one block. This is a contradiction since the existence shows that semisimple strata which intertwine have the same number of blocks. 
\end{proof}

As a useful consequence, we see that, given two semisimple strata which intertwine, we can find equivalent semisimple strata with elements which are conjugate. 

\begin{corollary}\label{corSameCharPolynomialForBetas}
Suppose that the semisimple strata~$[\Lambda,q,r,\beta]$ and~$[\Lambda',q,r,\beta']$ intertwine and that~$\Lambda,\Lambda'$ have the same period, and let~$\zeta: I\ra I'$ be the matching between their index sets. Then there are semisimple strata~$[\Lambda,q,r,\tilde{\beta}]$ and~$[\Lambda',q,r,\tilde{\beta}']$, equivalent to, and with the same associated splitting as,~$[\Lambda,q,r,\beta]$ and~$[\Lambda',q,r,\beta']$ respectively, such that~$\tilde{\beta}'_{\zeta(i)}$ has the same characteristic polynomial as~$\tilde{\beta}_i$, for all indices~$i\in I$. 
\end{corollary}

\begin{proof}
This follows immediately from Lemma~\ref{lemConjblocksSemisimple} (whose hypotheses are satisfied, thanks to Proposition~\ref{propMatching}) by putting, in the notation of the Lemma,~$\tilde\beta=\gamma$ and~$\tilde{\beta}'=g^{-1}\gamma g$.
\end{proof}

If~$[\Lambda,q,r,\beta]$ and~$[\Lambda',q',r',\beta']$ are strata in spaces~$V$ and~$V'$ respectively, then we put
\[
I([\Lambda^i,q,r,\beta],[\Lambda^{\zeta (i)},q,r,\beta'])) =\left \{g\mid V\xrightarrow{\sim} V' \mid g(\beta+\mfa_{-r})g^{-1} \cap (\beta'+\mfa'_{-r'}\ne\emptyset\right\}.
\]
This generalises the notion of intertwining and we say that any element of this set intertwines~$[\Lambda,q,r,\beta]$ with~$[\Lambda',q',r',\beta']$.

\begin{corollary}\label{corIntertwiningtwosemisimplestrataSameLevel}
Suppose that the semisimple strata~$[\Lambda,q,r,\beta]$ and~$[\Lambda',q,r,\beta']$ intertwine and that~$\Lambda,\Lambda'$ have the same period, and let~$\zeta: I\ra I'$ be the matching between their index sets. Then the intertwining set~$I([\Lambda,q,r,\beta],[\Lambda',q,r,\beta'])$ is equal to 
\[
(1+\mf{m}'_{-(r+k'_0)})(\prod_i I([\Lambda^i,q,r,\beta],[\Lambda^{\zeta (i)},q,r,\beta']))(1+\mf{m}_{-(r+k_0)}).
\]
\end{corollary}

\begin{proof}
By Corollary~\ref{corSameCharPolynomialForBetas} we may replace the strata with equivalent strata such that there is~$g\in\tilde G$ such that~$\beta'=g^{-1}\beta g$. The result now follows by applying Theorem~\ref{thmIntertwiningTwoSemisimpleStratumSameBeta} to the strata~$[\Lambda,q,r,\beta]$ and~$[g\Lambda',q,r,\beta]$.
\end{proof}

\subsection{For classical groups}

We continue with the notation from the previous section but assume now that all our strata are skew. We will prove the following strengthening of Proposition~\ref{propMatching} in this case.

\begin{proposition}\label{propIntertwiningtwoSkewsemisimplestrataSameLevel}
Suppose that~$[\Lambda,q,r,\beta]$ and~$[\Lambda',q,r,\beta']$ are two skew-semisimple strata which intertwine in~$G$ and let~$\zeta:I\to I'$ be the matching given by Proposition~\ref{propMatching}. Then:
\be
\item~$(V^i,h|_{V^i})\cong (V'^{\zeta(i)},h|_{V'^{\zeta(i)}})$, for all~$i\in I$; 
\item~the intertwining set~$I_G([\Lambda,q,r,\beta],[\Lambda',q,r,\beta'])$ is equal to 
\[
(G\cap (1+\mf{m}'_{-(r+k'_0)}))(G\cap \prod_i I_i)(G\cap (1+\mf{m}_{-(r+k_0)})),
\]
where~$I_i=I([\Lambda^i,q,r,\beta],[\Lambda'^{\zeta (i)},q,r,\beta']))$.
\ee
\end{proposition}

\begin{remark}\label{rempropIntertwiningtwoSkewsemisimplestrataSameLevelifollowsii}
Part~(ii) of Proposition~\ref{propIntertwiningtwoSkewsemisimplestrataSameLevel} part~(i) because
if~(i) is true then, by conjugating, we can assume that~$V^i=V'^{\zeta(i)}$, for all~$i\in I$. Then~(ii) follows from Corollary~\ref{corIntertwiningtwosemisimplestrataSameLevel} and a simple cohomology argument as in~\cite[Corollary~4.14]{stevens:05}. 
\end{remark}

As an immediate consequence of Proposition~\ref{propIntertwiningtwoSkewsemisimplestrataSameLevel} and the simple Skolem--Noether Theorem~\ref{thmSkolemNoether}, we get a Skolem--Noether Theorem for semisimple strata.

\begin{theorem}\label{thmSkolemNoetherSemisimple}
Let~$[\Lambda,q,r,\beta]$ and~$[\Lambda',q,r,\beta']$ be two skew-semisimple strata which intertwine in~$G$, and suppose that~$\beta$ and~$\beta'$ have the same characteristic
polynomial. Then there is an element~$g \in G$ such that~$g\beta g^{-1}$ is equal to~$\beta'$.
\end{theorem}

For the proof of Proposition~\ref{propIntertwiningtwoSkewsemisimplestrataSameLevel} we need the following idempotent lifting lemma.

\begin{lemma}\label{lemFindIdempotents}
Let~$(\mf{k}_r)_{r\ge 0}$ be a decreasing sequence of~$o_F$-lattices in~$A$ such that~$\mf{k}_r\mf{k}_s\subseteq\mf{k}_{r+s}$, for all~$r,s\ge 0$, and~$\bigcap_{r\ge 1}\mf{k}_r = \{0\}$. Suppose there is an element~$\alpha$ of~$\mf{k}_0$ which satisfies~$\alpha^2-\alpha\in \mf{k}_r$. Then there is an idempotent~$\tilde{\alpha}\in \mf{k}_0$ such that~$\tilde{\alpha}-\alpha\in \mf{k}_r$. Moreover, if~$\sigma(\alpha)=\alpha$ then we can choose~$\tilde{\alpha}$ such that~$\sigma(\tilde\alpha)=\tilde\alpha$. 
\end{lemma}

\begin{proof}
We define~$e_1:=\alpha$, and put~$e_2:=3e_1^2-2e_1^3\in\mf{k}_0$. A straightforward calculation shows that
\[
e_2^2-e_2 = 4(e_1^2-e_1)^3-3(e_1^2-e_1)^2 \in \mf{k}_{2r}.
\]
Continuing this process, we construct a sequence~$(e_n)_{n\ge 1}$ in~$\mf{h}_0$ which satisfies
\be
\item~$e_n\equiv e_i\pmod{\mf{k}_{2^ir}}$ and 
\item~$e_n\equiv e_n^2\pmod{\mf{k}_{2^nr}}$, 
\ee
for all positive integers~$i<n$. This sequence has a limit~$\tilde{\alpha}$ in~$\mf{b}_0$ which is, by construction, an idempotent congruent to~$\alpha$ modulo~$\mf{k}_r$. Moreover, by construction the sequence~$(e_n)$ is symmetric if~$\alpha$ is, in which case the limit~$\tilde{\alpha}$ is also symmetric. 
\end{proof}

We also need the classical group analogue of Lemma~\ref{lemConjblocksSemisimple}.

\begin{lemma}\label{lemConjblocksSkewSemisimple}
Let~$[\Lambda,q,r,\beta]$ and~$[\Lambda',q,r,\beta']$ be skew-semisimple strata which intertwine in~$G$ and let~$\zeta:I\to I'$ be the matching given by Proposition~\ref{propMatching}. Suppose moreover that~$(V^i,h|_{V^i})\cong (V'^{\zeta(i)},h|_{V'^{\zeta(i)}})$, for all~$i\in I$. Then there are an element~$g\in G$ and a skew element~$\gamma\in \prod_i A^{i,i}$ such that~$V^i=gV'^{\zeta(i)}$, for all indices~$i\in I$, and:
\begin{itemize}
 \item~$[g\Lambda',q,r,g\beta'g^{-1}]$ is equivalent to~$[g\Lambda',q,r,\gamma]$;
 \item~$[\Lambda,q,r,\beta]$ is equivalent to~$[\Lambda,q,r,\gamma]$; and
 \item~$[g\Lambda',q,r,\gamma]$ and~$[\Lambda,q,r,\gamma]$ are skew-semisimple strata with the same associated splitting~$V=\bigoplus_{i\in I}V^i$.
\end{itemize}
\end{lemma}

\begin{proof}
The proof is the same as that of Lemma~\ref{lemConjblocksSemisimple}. We only need to note that, once we have found~$\gamma_i$ and~$\gamma'_{\zeta(i)}$ with the same irreducible minimal polynomial then there is an element~$g\in G$ such that~$V^i=gV'^{\zeta(i)}$, since~$(V^i,h|_{V^i})\cong (V'^{\zeta(i)},h|_{V'^{\zeta(i)}})$, and then the elements~~$g\gamma'_{\zeta(i)} g^{-1}$ and $\gamma_i$ are conjugate in~$G_i=G\cap A^{i,i}$ by Remark~\ref{rempropIntertwiningtwoSkewsemisimplestrataSameLevelifollowsii} and Theorem~\ref{thmSkolemNoether}.
\end{proof}

\begin{proof}[Proof of Proposition~\ref{propIntertwiningtwoSkewsemisimplestrataSameLevel}]
It is sufficient to prove~(i) by Remark~\ref{rempropIntertwiningtwoSkewsemisimplestrataSameLevelifollowsii}.
We prove~(i) by strata induction, giving first the inductive step. Suppose that~$\Delta=[\Lambda,q,r,\beta]$ and~$\Delta'=[\Lambda',q,r,\beta']$ are skew-semisimple strata as in the proposition which intertwine in~$G$. Then the stratum~$[\Lambda,q,r+1,\beta]$ is equivalent to a skew-semisimple stratum~$\Delta_\gamma=[\Lambda,q,r+1,\gamma]$ whose splitting is a coarsening of that of~$\Delta$, by Theorem~\ref{thmDiagonalizedSimple}; similarly we have a skew-semisimple stratum~$\Delta'_\gamma=[\Lambda',q,r+1,\gamma']$. Since the strata~$\Delta_\gamma$,~$\Delta'_\gamma$ intertwine, we may apply Lemma~\ref{lemConjblocksSkewSemisimple} and, replacing~$\Delta'$ by its conjugate~$g\Delta'$, we may assume~$\gamma'=\gamma$.

Now, if~$h\in G$ intertwines the strata~$\Delta_\gamma$ and~$\Delta'_\gamma$ then Theorem~\ref{thmIntertwiningTwoSemisimpleStratumSameBeta} allows us to write~$h=xby$, with~$b\in B\cap G$ (and~$x,y$ in certain compact subgroups). Then Proposition~\ref{propS2}, applied as in the proof of Proposition~\ref{propMatching}, implies that the derived strata~$[\Lambda,r+1,r,s_\gamma(\beta-\gamma)]$ and~$[\Lambda',r+1,r,s_\gamma(\beta'-\gamma)]$ intertwine. On the other hand, these derived strata are equivalent to skew-semisimple strata so the base step (below) now implies that the bijection~$\zeta: I\ra I'$ has the property that~$(V^i,h_i^{\phi_i})\cong (V'^{\zeta(i)},h'^{\phi'_{\zeta(i)}}_{\zeta(i)})$, where~$\phi_i:F[\gamma_i]\to A^{i,i}$ is the embedding given by the splitting, and~$h_i^{\phi_i}$ is such that~$h_i=\lambda_i\circ h_i^{\phi_i}$ and~$h_i=h|_{V^i}$. 
But then we also have~$(V^i,h_i)\cong (V'^{\zeta(i)},h'_{\zeta(i)})$, as required.

It remains to show the base case~$r=q-1$. The lattice sequences have the same level, and so the same period, by Proposition~\ref{propFundamental}. Since they intertwine, the strata have the same characteristic polynomial. If two minimal strata with the same level intertwine and one of them is zero, then the other is zero, by Proposition~\ref{propFundamental}, and both have the trivial associated splitting. Thus we need only consider non-zero semisimple strata.

Replacing~$[\Lambda',q,q-1,\beta']$ be a conjugate if necessary, we may assume that the strata are intertwined by~$1$ so that 
\[
(\beta+\mf{a}_{1-q})\cap (\beta'+\mf{a}'_{1-q}) \neq \emptyset.
\]
Thus there is are elements~$a\in\mfa_1$ and~$a'\in\mfa'_1$ such that 
\[
z:=y_{\beta}+a=y_{\beta'}+a'.
\] 
By the bijectivity of~$\zeta$ we can assume that~$I=I'$ and~$\zeta$ is the identity. Let~$i\in I$.

We show that there is an idempotent~$e$ such that~$e\equiv 1^i\pmod{\mfa_1}$ and~$e\equiv 1'^i\pmod{\mfa'_1}$: There is a polynomial~$Q\in o_F[X]$ such that~$Q(y_{\beta})\equiv 1^i\pmod{\mf{a}_1}$. Moreover, by replacing~$Q(X)$ by~$\frac 12(Q(X)+\sigma(Q)(\pm X))$, we can choose~$Q$ such that, 
for all~$j$, the coefficient of~$X^j$ is symmetric (resp. skew-symmetric) if and only if~$y^j_\beta$ is symmetric (resp. skew-symmetric). We have a canonical isomorphism from~$\kappa[\bar{y}_{\beta}]$ to~$\kappa[\bar{y}_{\beta'}]$ (mapping~$\bar{y}_{\beta}$ to~$\bar{y}_{\beta'}$) so~$Q(y_{\beta'})$ is congruent to some idempotent modulo~$\mfa'_1$, and indeed~$Q(y_{\beta'})\equiv 1'^i\pmod{\mf{a}'_1}$ since the matching~$\zeta$ is given by matching minimal polynomials. By Proposition~\ref{lemFindIdempotents} applied with~$\mf{b}_0=\mfa_0$ and~$\mf{k}_r=\mfa_r\cap\mfa'_r$, there is a symmetric idempotent~$e\in \mf{a}_0\cap \mf{a}'_0$ congruent to~$Q(z)$ modulo both radicals.

The idempotent~$e$ gives a new splitting~$V=\tilde{V}^i\oplus (\tilde{V}^i)^\perp$ for \emph{both} lattice sequences. 

Finally, we show that~$V^i$ and~$\tilde{V^i}$ are isomorphic signed hermitian spaces.  
We define the map~$\psi: V^i\ra \tilde{V}^i$ to be the restriction of~$e$ to~$V^i$. We first show that the map is injective. If~$v$ is a non-zero element of its kernel, then there is an integer~$l$ such that~$v\in \Lambda^i_l\setminus \Lambda^i_{l+1}$. But then
\[
0\not\equiv v\equiv 1^iv\equiv ev\equiv 0 \pmod{\Lambda^i_{l+1}}, 
\] 
where the third congruence uses that~$e\equiv 1^i\pmod{\mf{a}_1}$. Similarly, the restriction of~$1^i$ to~$\tilde{V}^i$ is injective and these maps induce pairwise inverse~$\kappa_F$-isomorphisms between~$\Lambda^i_l/\Lambda^i_{l+1}$ and~$\tilde{\Lambda}^i_l/\tilde{\Lambda^i_{l+1}}$ where 
$\tilde{\Lambda}^i_l$ is the intersection of~$\Lambda_l$ with~$\tilde{V}^i$. Thus~$\psi(\Lambda^i)$ is equal to~$\tilde{\Lambda}^i$. 

We now compare the hermitian structures. For~$v\in \Lambda^i_l$ and~$w\in (\Lambda^i_{l+1})^\#$ we have 
\[
\overline{h(v,w)}
=\overline{h(v,1^iw)}
=\overline{h(v,ew)}
=\overline{h(v,e^2w)}
=\overline{h(ev,ew)}
=\overline{h(\psi(v),\psi(w))}
\]
By Proposition~\ref{propliftResidualIsometries} there is an~$F$-linear isometry 
\[
\tilde{\psi}:(V^i,h|_{V^i})\ra (\tilde{V}^i,h|_{\tilde{V}^i})
\] 
such that
\bi
\item $\tilde{\psi}(\Lambda^i)=\tilde{\Lambda}^i$ and 
\item $\tilde{\psi}$ and~$\psi$ induce the same isomorphism on~$\Lambda^i_l/\Lambda^i_{l+1}$, for all integers~$l$.
\ei
Thus,~$V^i$,~$\tilde{V}^i$, and similarly~$V'^i$, are isomorphic signed hermitian spaces. 
\end{proof}

\subsection{Matching for equivalent strata}
We need to understand the matching between two equivalent strata, and for that reason we have the following three results. 

\begin{lemma}\label{lemCentralIdempoents}
Let~$[\Lambda,q,m,\beta]$ be a semisimple sratum and assume that~$e$ is an idempotent in~$\prod_iA^{i,i}$ such that every non-zero element~$x$ of~$\prod_iA^{i,i}$ satisfies~$\nu_\Lambda(ex-xe)>\nu_\Lambda(a)$. Then~$e$ is a central idempotent of~$\prod_iA^{i,i}$, i.e. commutes with all elements of~$\prod_iA^{i,i}$. 
\end{lemma}

\begin{proof}
Let~$x$ be an element of~$\prod_iA^{i,i}$ and put~$x'=a_e(x)=ex-xe$. Then~$ex'e=0$ and one checks that~$a_e(a_e(x'))=a_e(x')$. The condition on~$e$ now implies that~$a_e(x')=0$ is zero and thus~$ex'=ea_e(x')=0=a_e(x')e=x'e$. This implies that~$ex=exe=xe$ and, since~$x$ was arbitrary,~$e$ is central. 
\end{proof}

\begin{lemma}\label{lemComparingIntertwinings}
Let~$[\Lambda,q,m,\beta]$ and~$[\Lambda,q,m,\beta']$ be two semisimple strata, such that  
\[ 
\mf{a}_1B_{\beta'} \mf{a}_1=\mf{a}_1B_\beta \mf{a}_1
\] 
then there are a bijection~$\tau: I\ra I'$ and an element~$g$ of~$\tilde{U}^1(\mf{a})$ such that 
\begin{enumerate}
 \item~$1^i\equiv 1^{\tau(i)} \pmod{\mf{a}_1}$, and 
 \item~$g1^ig^{-1}=1^{\tau(i)}$
\end{enumerate}
for all indices~$i\in I$. Moreover, the bijection~$\tau$ satisfies
\[
\dim_{\kappa_F}(\Lambda^i_j/\Lambda^i_{j+1})=\dim_{\kappa_F}(\Lambda'^{\tau(i)}_j/\Lambda'^{\tau(i)}_{j+1}),\ \text{ for all }i\in I,\ j\in\bbZ.
\]
\end{lemma}

\begin{proof}
By the equality of the two sets and Lemmas~\ref{lemFindIdempotents} and~\ref{lemCentralIdempoents}, every  primitive central idempotent of~$B_\beta$ has to be congruent modulo~$\mf{a}_1(\Lambda)$ to a sum of primitive central idempotents of~$B_{\beta'}$, and vice versa. The first part follows from this. For the second part, take the map~$g$ which sends~$v\in V$ to~$\sum_i1^{\tau(i)}1^iv$. Finally, the map~$v\mapsto 1^iv$ induces, for each~$j\in\bbZ$, a linear map~$\Lambda'^{\tau(i)}_j/\Lambda'^{\tau(i)}_{j+1}\ra \Lambda^i_j/\Lambda^i_{j+1}$ whose inverse is induced by~$v\mapsto 1^{\tau(i)} v$.
\end{proof}

\begin{lemma}\label{lemEqSesiStrata}
Let~$[\Lambda,q,m,\beta]$ and~$[\Lambda,q,m,\beta']$ be equivalent semisimple strata. Then there is an element~$g$ of~$1+\mf{m}_{-(k_0+m)}(\beta,\Lambda)$ such that~$[\Lambda,q,m,\beta]$ and~$[\Lambda,q,m,g\beta'g^{-1}]$ have the same associated splitting. 
\end{lemma}

\begin{proof}
Note that we may replace~$[\Lambda,q,m,\beta]$ by an equivalent stratum with the same splitting (and likewise for~$[\Lambda,q,m,\beta']$). Thus, by applying Corollary~\ref{corSameCharPolynomialForBetas}, we may assume that~$\beta$ and~$\beta'$ have the same characteristic polynomial and thus there is an element~$x$ of~$\tilde{G}$ such that~$x\beta x^{-1}=\beta'$. Note that this implies that~$xV^i=V'^{\zeta(i)}$.

Since the strata intertwine, Proposition~\ref{propMatching} gives us a matching~$\zeta:I\ra I'$ such that the minimal polynomials satisfy~$\mu_{\beta_i}=\mu_{\beta'_{\zeta(i)}}$ and~$\dim_FV^i= \dim_F V^{\zeta(i)}$, for each~$i\in I$. We can also compare the intertwining sets of the strata (which are equal) and then Lemma~\ref{lemComparingIntertwinings} gives us a map~$\tau: I\ra I'$ such that~$1^i\equiv 1^{\tau(i)}\pmod{\mf{a}_1}$, for all~$i\in I$. Since the identity intertwines the two strata, Corollary~\ref{corIntertwiningtwosemisimplestrataSameLevel} implies that we can write the identity as~$uyv$, with~$u,v\in\tilde U^1(\Lambda)$ and~$y=\prod_{i\in I} y_i$ such that~$y_iV^i=V'^{\zeta(i)}$. Moreover, we have~$y=u^{-1}v^{-1}\in\tilde U^1(\Lambda)$ so that~$y_i\Lambda^i=\Lambda'^{\zeta(i)}$. Thus
\[
1^i \equiv y 1^i y^{-1}=1^{\zeta(i)} \pmod{\mfa_1}.
\]
In particular, we get~$1^{\tau(i)}\equiv 1^{\zeta(i)} \pmod{\mfa_1}$ so that~$\zeta=\tau$, and then Lemma~\ref{lemComparingIntertwinings} also implies that~$\zeta$ satisfies the extra condition
\[
\dim_{\kappa_F}(\Lambda^i_j/\Lambda^i_{j+1})=\dim_{\kappa_F}(\Lambda'^{\zeta(i)}_j/\Lambda'^{\zeta(i)}_{j+1}),\ \text{ for all }i\in I,\ j\in\bbZ.
\]
Now~$\Lambda^i$ and~$x^{-1}\Lambda'^{\zeta(i)}$ are~$o_{F[\beta_i]}$-lattice sequences in~$V^i$ with successive quotients of the same dimensions so there is an element~$z_i$ of~$B_{\beta_i}^\times$ such that~$z_i\Lambda^i=x^{-1}\Lambda'^{\zeta(i)}$. In particular, writing~$z=\prod_{i\in I}z_i$, the element~$xz$ conjugates~$\beta$ to~$\beta'$ and lies in~$\tilde{U}(\Lambda)$.

Finally, since the strata~$[\Lambda,q,m,\beta]$ and~$[\Lambda,q,m,\beta']$ be equivalent, the element~$xz$ also lies in~$\mf{n}_{-m}(\beta,\Lambda)\cap\tilde{U}(\Lambda)$ which, by Lemma~\ref{lemDescrmfn}, is~$(1+\mf{m}_{-(k_0+m)}(\beta,\Lambda))\mf{b}_0^\times$. Hence we can write~$xz=gb$, with~$b\in \mf{b}_0^\times$ and~$g\in 1+\mf{m}_{-(k_0+m)}(\beta,\Lambda)$. 
\end{proof}

\section{Intertwining and conjugacy for semisimple strata}

In the case of simple strata on a fixed lattice \emph{chain}, intertwining implies conjugacy (see~\cite[2.6.1]{bushnellKutzko:93}). The same result is true for arbitrary lattice sequences and, as we prove here, for simple skew strata (that is,~$G$-intertwining implies~$G$-conjugacy). However, the analogous result is no longer true for semisimple strata. As well giving some examples to illustrate this, we give a useful sufficient additional condition to guarantee that the strata are indeed conjugate.

\subsection{For general linear groups}

\begin{theorem}[{cf.~\cite[2.6.1]{bushnellKutzko:93}}]\label{thmIntConSimple}
Suppose~$[\Lambda,q,r,\beta]$ and~$[\Lambda,q,r,\beta']$ are simple strata which intertwine. Then, they are conjugate by an element of~$\tilde{\U}(\Lambda)$.
\end{theorem}

\begin{proof}
By Corollary~\ref{corSameCharPolynomialForBetas}, we can assume that~$\beta$ and~$\beta'$ have the same characteristic polynomial. By~\cite[Lemma~1.6]{bushnellHenniart:96}, there is an element of~$\tilde{\U}(\Lambda)$ which conjugates~$\beta$ to~$\beta'$. 
\end{proof}

In contrast to simple strata we cannot achieve intertwining implies conjugacy for semisimple strata.

\begin{example}
Let~$V$ be a~$4$-dimensional vector space over~$F$ with basis~$v_1,\ldots, v_4$ and let~$\Lambda$ be the lattice chain of period~$2$ such that 
\[
\Lambda_0=v_1o_F+v_2o_F+v_3o_F+v_4o_F,\qquad \Lambda_1=v_1o_F+v_2o_F+v_3o_F+v_4\mf{p}_F.
\]
Then, with respect to the basis,~$\mf{a}_0(\Lambda)$ is 
\[
\begin{pmatrix}
o_F&o_F&o_F&o_F\\
o_F&o_F&o_F&o_F\\
o_F&o_F&o_F&o_F\\
\mf{p}_F&\mf{p}_F&\mf{p}_F&o_F\\
\end{pmatrix}
\]
The two elements:
\[
b:=\diag(\pi^{-1},\pi^{-1},-\pi^{-1},-\pi^{-1}),\qquad 
b':=\diag(-\pi^{-1},-\pi^{-1},\pi^{-1},\pi^{-1})
\] 
give two semisimple strata~$[\Lambda,2,1,b]$ and~$[\Lambda,2,1,b']$ which intertwine but are not conjugate over~$\Aut_F(V)$. Indeed, suppose for contradiction that the strata are conjugate under an element of~$\tilde{G}$; then this element has to be an element of the normalizer of~$\Lambda$ and thus by Lemma~\ref{lemEqSesiStrata} we can assume after conjugation that the associated splittings, which are the same for both strata, are conjugated to each other. Note that this splitting is given by~$V^1=v_1F+v_2F$ and~$V^2=v_3F+v_4F$. The minimal polynomials of the strata force that the matching has to be given by exchanging the two blocks~$V^1$ and~$V^2$. But this is not possible, because the image of~$\Lambda^1=\Lambda\cap V^1$ contains only one homothety class of lattices, while the image of~$\Lambda^2=\Lambda\cap V^2$ contains two. 
\end{example}

Thus we impose an extra condition in the following Theorem.

\begin{theorem}\label{thmintConjSemisimple}
Suppose that~$[\Lambda,q,r,\beta]$ and~$[\Lambda,q,r,\beta']$ are two non-zero semisimple strata which intertwine and let~$\zeta$ be the matching between their index sets. Suppose moreover that 
\beq\label{eqConLambdaLambdaPrime}
\dim_{\kappa_F}(\Lambda^i_j/\Lambda^i_{j+1})=\dim_{\kappa_F}(\Lambda'^{\zeta(i)}_j/\Lambda'^{\zeta(i)}_{j+1}),\ \text{ for all }i\in I,\ j\in\bbZ.
\eeq
Then the strata are conjugate by an element of~$\tilde{U}(\Lambda)$. 
\end{theorem}

\begin{proof}
Fix an index~$i$ of~$I$. By Proposition~\ref{propMatching} and Corollary~\ref{corIntertwiningtwosemisimplestrataSameLevel} there is an~$F$-linear isomorphism~$g_i: V^i\ra V^{\zeta(i)}$ such that~$[g_i\Lambda^i,q,r,g_i\beta^ig_i^{-1}]$ intertwines~$[\Lambda'^{\zeta(i)},q,r,\beta'^{\zeta(i)}]$. Now condition~\eqref{eqConLambdaLambdaPrime} implies that the lattice sequences~$g_i\Lambda^i$ and~$\Lambda'^{\zeta(i)}$ are conjugate so, modifying~$g_i$ if necessary, we may assume~$g_i\Lambda^i_j=\Lambda'^{\zeta(i)}_j$, for each~$j\in\bbZ$. Now we can apply Theorem~\ref{thmIntConSimple} so that, replacing~$g_i$ by a translate by an element of~$\tilde U(\Lambda'^{\zeta(i)})$, we can assume that~$[g_i\Lambda^i,q,r,g_i\beta^ig_i^{-1}]$ is equivalent to~$[\Lambda'^{\zeta(i)},q,r,\beta'^{\zeta(i)}]$. Then
\[
\left(\prod_{i\in I} g_i\right) \Lambda_j = \bigoplus_{i\in I} g_i\Lambda^i_j = \bigoplus_{i\in I} \Lambda'^{\zeta(i)}_j = \Lambda_j.
\]
so that~$\prod_{i\in I} g_i\in\tilde U(\Lambda)$ conjugates the first to the second stratum.
\end{proof}

\subsection{For classical groups}

We give here the similar ``intertwining implies conjugacy'' statements for skew-semisimple strata, beginning with the simple case.

\begin{theorem}\label{thmIntConSkewSimple}
Suppose~$[\Lambda,q,r,\beta]$ and~$[\Lambda,q,r,\beta']$ are two skew-simple strata which intertwine in~$G$. Then they are conjugate over~$U(\Lambda)$.
\end{theorem}

\begin{proof}
The proof is the mutatis mutandis that of Theorem~\ref{thmIntConSimple}: we apply Corollary~\ref{corSameCharPolynomialForBetas}, then Theorem~\ref{thmSkolemNoether}, and then 
Proposition~\ref{propConjug}.
\end{proof}

As in the non-skew case, this is no longer true if one replaces simple by semisimple.

\begin{example}
Consider a ramified quadratic field extension~$F|F_0$ and a skew-hermitian form on~$V=F^4$ whose Gram matrix~$(h_{ij})$ with respect to the standard basis is the anti-diagonal matrix with entries
\[
h_{41}=h_{32}=-1=-h_{23}=-h_{14},
\]
and write~$G$ for the isometry group of this form. Let~$\pi$ be  a skew-symmetric uniformizer of~$F$ and let~$z$ be a non-square in~$F_0$. Let~$\Lambda$ be the self-dual lattice chain corresponding to the hereditary order
\[
\begin{pmatrix}
o_F&\mf{p}_F&\mf{p}_F&\mf{p}_F\\
o_F&o_F&\mf{p}_F&\mf{p}_F\\
o_F&o_F&o_F&\mf{p}_F\\
o_F&o_F&o_F&o_F\\
\end{pmatrix}
\]
We define the skew-symmetric elements:
\[
b:=\diag(\pi^{-1}z,\pi^{-1},\pi^{-1},z\pi^{-1}),\qquad 
b':=\diag(\pi^{-1},\pi^{-1}z,\pi^{-1}z,\pi^{-1}).
\] 
The minimal skew-semisimple strata~$[\Lambda,4,3,b]$ and~$[\Lambda,4,3,b']$ intertwine over~$G$  because~$b$ is conjugate to~$b'$ under~$\U(h)$, but the strata are not conjugate under~$G$ because they are not conjugate under~$\tilde U(\Lambda)$.
\end{example}

As an immediate consequence of Proposition~\ref{propIntertwiningtwoSkewsemisimplestrataSameLevel} and  Theorem~\ref{thmIntConSkewSimple} (as in the non-skew case above) we have:

\begin{theorem}\label{thmintConjSkewSemisimple}
Suppose that~$[\Lambda,q,r,\beta]$ and~$[\Lambda,q,r,\beta']$ are two non-zero skew-semisimple strata which intertwine in~$G$, with matching~$\zeta$, such that~\eqref{eqConLambdaLambdaPrime} holds. Then the strata are conjugate by an element of~$U(\Lambda)$. 
\end{theorem}


\section{Semisimple characters}

Associated to the semisimple strata studied in the previous sections, we have sets of characters of certain compact open subgroups, which are call \emph{semisimple characters}. The purpose of this section is both briefly to recall their definitions and properties (from~\cite{bushnellKutzko:93} and~\cite{stevens:05}) and to ensure that all the results we need are available for arbitrary lattice sequences.

\subsection{Semisimple characters for~$\tilde{G}$}

For fix a semisimple stratum~$[\Lambda,q,0,\beta]$. 
Define~$r:=-k_0(\beta,\Lambda)$ and let~$[\Lambda,q,r,\gamma]$ be a semisimple stratum equivalent to~$[\Lambda,q,r,\beta]$ such that~$\gamma$ commutes with the projections~$1^i$ of the associated splitting of~$\beta$. If~$[\Lambda,q,r,\beta]$ is minimal then we take~$\gamma$ to be zero. 

\subsubsection{The rings of a semisimple stratum (cf.~\cite[3.1]{bushnellKutzko:93})}
We start with the orders~$\mf{h}(\beta,\Lambda)\subseteq \mf{j}(\beta,\Lambda)$ defined inductively by 
\begin{itemize}
 \item~$\mf{h}(\beta,\Lambda)=\mf{b}_{\beta,0}+\mf{h}(\gamma,\Lambda)\cap \mfa_{\floor{\frac{r}{2}}+1}$,
 \item~$\mf{j}(\beta,\Lambda)=\mf{b}_{\beta,0}+\mf{j}(\gamma,\Lambda)\cap \mfa_{\floor{\frac{r+1}{2}}}$,
\end{itemize}
with~$\mf{h}(0,\Lambda)=\mf{j}(0,\Lambda)=\mfa_0$. 
We define now the groups
\[
H^{m+1}(\beta,\Lambda):=\mf{h}(\beta,\Lambda)\cap \tilde{\U}^{m+1}(\Lambda),\qquad 
J^{m+1}(\beta,\Lambda):=\mf{j}(\beta,\Lambda)\cap \tilde{\U}^{m+1}(\Lambda),
\] 
for~$m\geq -1$, and write~$H$ and~$J$ instead of~$H^0$ and~$J^0$. 

We now begin the proofs of the statements in~\cite[section~3.1]{bushnellKutzko:93} for semisimple strata. (Note that some of these are already in~\cite{stevens:05}.)

\begin{proposition}[{cf.~\cite[(3.1.9)]{bushnellKutzko:93}}]\label{propBK3.1.9}
\begin{itemize}
 \item[(i)] For all~$-1\leq t\leq r$, the lattice~$\mf{h}^{\floor{\frac{t}{2}}}(\beta,\Lambda)$ is a bimodule over the ring~$\mf{n}_{-t}(\beta,\Lambda)$.
 \item[(ii)] If~$r<n$, we have~$\mf{h}^{k}(\beta,\Lambda)$ is equal~$\mf{h}^{k}(\gamma,\Lambda)$ for~$k\geq \floor{\frac{r}{2}}+1$.
 \item[(iii)] For~$k\geq 0$,~$\mf{h}^{k}(\beta,\Lambda)$ is a~$\mf{b}_\beta$-bimodule.
 \item[(iv)]~$\mf{h}(\beta,\Lambda)$ is a ring and in particular an~$o_F$-order in~$A$ and~$\mf{h}^k(\beta,\Lambda)$ is a two-sided ideal of~$\mf{h}(\beta,\Lambda)$, for all non-negative integers~$k$.
 \item[(v)] Let~$t\leq r-1$ and let~$[\Lambda,q,t,\beta']$ be a semisimple stratum equivalent to~$[\Lambda,q,t,\beta]$. Then~$\mf{h}^k(\beta,\Lambda)$ is equal to~$\mf{h}^k(\beta',\Lambda)$, for all non-negative integers~$k>t-\floor{\frac{r+1}{2}}$.
\end{itemize}
\end{proposition}

\begin{proof}
In~\cite[(3.1.9)]{bushnellKutzko:93} the statement is proven for strict simple strata. In the case of a non-strict simple stratum~$[\Lambda,q,m,\beta]$, the stratum 
\[ 
[\bigoplus_{l=0}^{e(\Lambda)-1}(\Lambda-l),q,m,\beta^{\oplus e}]
\] 
is a strict simple stratum and, using the identity 
\[
1_V\mf{h}^{k}(\beta^{\oplus e},\oplus_l (\Lambda-l))1_V=\mf{h}^{k}(\beta,\Lambda)
\] 
(where~$1_V$ denotes projection onto the first copy of~$V$ in~$\bigoplus_{l=0}^{e(\Lambda)-1}V$) we get the result for all simple strata. Thus we continue with the case of semisimple strata. 

We begin with the proof of~(v), but only for the case where the strata in~(v) have the same associated splitting; we prove the general case after the next four lemmas. We use the idea of~\cite[Lemma~3.9]{stevens:05}. Assume that~$\mf{h}^k(\beta,\Lambda)$ and~$\mf{h}^k(\beta',\Lambda)$ are defined using the same~$\gamma$ and that~$[\Lambda,q,t,\beta']$ has the same associated splitting as~$[\Lambda,q,t,\beta]$. In particular, we immediately get that~$\mf{h}^k(\beta,\Lambda)\cap A^{i,j}=\mf{h}^k(\beta',\Lambda)\cap A^{i,j}$, for~$i\ne j$, from the definition, while~$\mf{h}^k(\beta,\Lambda)\cap A^{i,i}=\mf{h}^k(\beta',\Lambda)\cap A^{i,i}$ follows from the simple case.

Before proving~(v) in general, we show how the remaining assertions follow from it. (ii) is straightforward while induction and (v) imply that the definition of~$\mf{h}(\beta,\Lambda)$ does not depend on the choice of~$\gamma$. Now~(i) follows from~\cite[3.10(ii)]{stevens:05}, (iv) follows by induction from~(i) and~(ii), and finally~(iii) follows from~(iv) and~(i).

To finish the proof of Proposition~\ref{propBK3.1.9}(v) we need the following sequence of lemmas.

\begin{lemma}\label{lemHisNmoduleForBigN}
$\mf{h}^{\max\{0,1+t-\floor{\frac{r+1}{2}}\}}$ is an~$\mf{n}_{-t}(\beta,\Lambda)$-bimodule for all~$r\geq t\geq 0$.
\end{lemma}

For this we need the analogue of~\cite[Lemma~3.10]{stevens:05} for~$\mf{h}$ instead of~$\mf{j}$ (see the sentence following \emph{loc. cit.}).

\begin{lemma}[{cf.~\cite[Lemma~3.10(i)]{stevens:05}}]\label{lemStevens310i}\ 
\begin{enumerate}
 \item For all integers~$k< \frac{r}2$, we have~$\mf{n}_{-k}\cap\mf{a}_{r-k}\subseteq \mf{h}^{r-k}(\beta,\Lambda)$.
 \item For all integers~$k\le \frac{r}2$, we have~$\mf{n}_{-k}\cap\mf{a}_{r-k}\subseteq \mf{j}^{r-k}(\beta,\Lambda)$
\end{enumerate}
\end{lemma}

\begin{proof}[Proof of Lemma~\ref{lemHisNmoduleForBigN}]
The proof is by induction on~$r=-k_0(\beta,\Lambda)$. We have the two important identities:
 \begin{align*}
\mf{h}^t(\beta,\Lambda)=\mf{b}_{\beta,t}+\mf{h}^{\max\{t,\floor{\frac{r}{2}}+1\}}(\gamma,\Lambda) 
 \end{align*}
 and 
 \begin{align*}
\mf{n}_{-t}=\mf{b}_{\beta,0}+\mf{n}_{-t}(\beta,\Lambda)\cap\mf{a}_{r-t}
\end{align*}
We write~$t_0$ for~$\max\{0,1+t-\floor{\frac{r+1}{2}}\}$ so that
\begin{equation}\label{eq1}
 t_0+r-t\geq 1+\floor{\frac{r}{2}} \text{ and }2(r-t+t_0)>r.
\end{equation}
We have to show that~$\mf{n}_{-t}\mf{h}^{t_0}(\beta,\Lambda)$ is a subset of~$\mf{h}^{t_0}(\beta,\Lambda)$. We have
\begin{itemize}
 \item $\mf{h}^{\max\{t_0,\floor{\frac{r}{2}}+1\}}(\gamma,\Lambda)=\mf{h}^{\max\{t_0,\floor{\frac{r}{2}}+1\}}(\beta,\Lambda)$ is a~$\mf{b}_{\beta,0}$-module.
 \item $(\mf{n}_{-t}(\beta,\Lambda)\cap\mf{a}_{r-t})\mf{b}_{\beta,t_0}$ is contained in~$\mf{n}_{t_0-t}(\beta,\Lambda)\cap\mf{a}_{t_0+r-t}$, which is a subset of $\mf{h}^{\max\{t_0,\floor{\frac{r}{2}}+1\}}(\gamma,\Lambda)$ by~\eqref{eq1} and Lemma~\ref{lemStevens310i}.
\end{itemize}
The last containment we need, that~$\mf{n}_{-t}(\beta,\Lambda)\cap\mf{a}_{r-t}\mf{h}^{\max\{t,\floor{\frac{r}{2}}+1\}}(\gamma,\Lambda)$ is a subset 
of~$\mf{h}^{t_0}(\beta,\Lambda)$, is proved by induction. The case of~$\gamma=0$ is trivial, while the induction step is a result of the equality~$\mf{n}_{-t}(\beta,\Lambda)\cap\mf{a}_{r-t}=\mf{n}_{-t}(\gamma,\Lambda)\cap\mf{a}_{r-t}$ and the induction hypothesis. 
\end{proof}

Finally, we see that the proof of the general case of Proposition~\ref{propBK3.1.9}(v) follows from Lemmas~\ref{lemHisNmoduleForBigN} and~\ref{lemEqSesiStrata}. 
\end{proof}

Given now the preliminary results on semisimple strata that we have obtained in previous sections and Proposition~\ref{propBK3.1.9}, we can follow the definitions and proofs of~\cite[section~3.1]{bushnellKutzko:93}, from~(3.1.3) to~(3.1.21), to see that if one makes the obvious substitution
\begin{itemize}
 \item ``replace~$\mf{b}_{{\beta},t}\mf{n}_l$ by~$\mf{a}_t\cap \mf{n}_{t+l}$'',
\end{itemize}
then everything is true except possibly for the equalities in~(3.1.9)(iii), in~(3.1.10)(iii) and in~(3.1.11). (Some of these are already described in~\cite{stevens:05}.) Thus, from now on, we will use these statements from~\cite{bushnellKutzko:93} for semisimple strata by referring to~\cite{bushnellKutzko:93} (and giving the reference to~\cite{stevens:05} if there is one).

\subsubsection{Characters (cf.~\cite[3.2]{bushnellKutzko:93})}	
Here we introduce the semisimple characters and their groups exactly the same way as it was done in~\cite[section~3.2]{bushnellKutzko:93} for simple characters. This definition is equivalent to the definition given in~\cite[section~3]{stevens:05}. We fix an additive character~$\psi_F$ of~$F$ of level one (that is, trivial on~$\mf{p}_F$ but not on~$o_F$). We define~$\psi_A:=\psi_F\circ \trace_{A|F}$ and a character
\[
\psi_{\beta}:\tilde{\U}_{\floor{\frac{n}{2}}+1}(\Lambda)\rightarrow \mathbb{C}^\times,\quad \psi_{\beta}(1+x):=\psi_A(\beta x).
\]
The kernel of~$\psi_\beta$ contains~$\tilde{\U}_{n+1}(\Lambda)$ because~$\psi_F$ has level one. 

\begin{definition}
If~$q=-k_0(\beta,\Lambda)$ then we define the set~$\C(\Lambda,m,\beta)$ to be the set of all characters~$\theta: H^{m+1}(\beta,\Lambda)\ra \mathbb{C}$ such that:
\begin{enumerate}
 \item the restriction of~$\theta$ to~$H^{m+1}(\beta,\Lambda)\cap\tilde{U}^{\floor{\frac{q}{2}+1}}(\Lambda)$ is equal to~$\psi_{\beta}$;
 \item the restriction of~$\theta$ to~$H^{m+1}\cap B^\times_{\beta}$ factors through the determinant map~$\det_{B_\beta}: B^\times_\beta\ra F[\beta]^\times$.
\end{enumerate}

If~$q>-k_0(\beta,\Lambda)$ then we define~$\C(\Lambda,m,\beta)$ inductively to be the set of all  characters~$\theta: H^{m+1}(\beta,\Lambda)\ra \mathbb{C}$ such that:
\begin{enumerate}
 \item~$\theta$ is normalized by~$\mf{n}(\Lambda)\cap B_\beta^\times$;
 \item the restriction of~$\theta$ to~$H^{m+1}\cap B^\times_{\beta}$ factors through the determinant 
map~$\det_{B_\beta}: B^\times_\beta\ra F[\beta]^\times$;
 \item if~$m'=\max(m,\floor{\frac{r}{2}})$, then~$\theta|_{H^{m'+1}(\beta,\Lambda)}$ is of the form~$\theta_0\psi_c$, for some~$\theta_0\in \C(\Lambda,m',\gamma)$ and~$c=\beta-\gamma$.
  \end{enumerate}
\end{definition}

\begin{remark}
For minimal~$\beta$ and~$m\geq\floor{\frac{q}{2}}$ we have~$\C(\Lambda,m,\beta)=\{\psi_\beta\}$
\end{remark}

We write~$I_H(\theta,\theta')$ for the intertwining of two characters in a group~$H$, i.e.~$g\in H$ is an element of~$I_H(\theta,\theta')$ if and only if~$\theta^g:x\mapsto \theta(gxg^{-1})$ and~$\theta'$ agree on the intersection of their domains. In the case~$H=\tilde{G}$ we omit the subscript.

Let us first recall the intertwining formula for a semisimple character:

\begin{proposition}[{\cite[3.22]{stevens:05},~\cite[(3.3.2)]{bushnellKutzko:93}}]\label{propIntertwiningSemisimplCharacter}
The intertwining of a semi\-simple character~$\theta\in \C(\Lambda,m,\beta)$ is given by~$S(\beta)B^\times_\beta S(\beta)$ where 
\[
S(\beta)=S(\Lambda,m,\beta)=1+\mf{m}_{-k_0-m}+\mf{j}^{\floor{\frac{-k_0+1}{2}}}(\beta,\Lambda).
\]
\end{proposition}

Reading the proofs of~\cite{bushnellKutzko:93} from~(3.2.1) to~(3.5.10) we see that all statements are true for semisimple characters (after replacing~$\mf{b}_{{\beta},t}\mf{n}_l$ by~$\mf{a}_t\cap \mf{n}_{t+l}$ throughout) except the statement~(3.5.1). However, there is an obvious modification of~(3.5.1) which is still true:

\begin{proposition}[{see~\cite[(3.5.1)]{bushnellKutzko:93}}]\label{propBK3.5.1}
Let~$[\Lambda,q,m,\beta]$ and~$[\Lambda,q,m,\beta']$ be semisimple strata such that~$\C(\Lambda,m,\beta)\cap \C(\Lambda,m,\beta')$ is non-empty. Then there is a bijection~$\tau: I\ra I'$ such that~$1^i\equiv 1^{\tau(i)}\pmod{\mf{a}_{1}}$, and:
\begin{enumerate}
 \item~$k_0(\beta,\Lambda)=k_0(\beta',\Lambda)$;
 \item the field extensions~$E_i|F$ and~$E'_{\tau(i)}|F$ have the same inertia degree and the same ramification index;
 \item the dimensions of~$V^i$ and~$V'^{\tau(i)}$ as~$F$-vector spaces coincide;
 \item there is an element~$g$ of~$S(\beta)$ such that~$gV^i$ is equal to~$V'^{\tau(i)}$. 
\end{enumerate}
\end{proposition}

\begin{proof}
The existence of~$\tau$ follows from the two descriptions of the intertwining of an element~$\theta\in \C(\Lambda,m,\beta)\cap \C(\Lambda,m,\beta')$,
\[
I(\theta)=S(\beta)B_\beta S(\beta)=S(\beta')B_{\beta'} S(\beta'),
\]
together with Lemma~\ref{lemComparingIntertwinings}. We now follow the proof of~\cite[(3.5.1)]{bushnellKutzko:93} to get that~$\mf{b}_{\beta,0}/\mf{b}_{\beta,1}$ is isomorphic to~$\mf{b}_{\beta',0}/\mf{b}_{\beta',1}$, by an isomorphism of~$\kappa_F$-algebras which maps~$1^i$ to~$1^{\tau(i)}$. We also have that~$(1^i\mf{a}_0 1^i)/\mf{a}_1=(1^{\tau(i)}\mf{a}_0 1^{\tau(i)})/\mf{a}_1$ and thus, as in the proof of~\cite[(2.1.4)]{bushnellKutzko:93}, we get the desired equalities. 

The equality of the additive closures of the intertwining set~$I(\theta)\cap\tilde U(\Lambda)$ in terms of~$\beta$ and~$\beta'$ 
implies that, for each~$i\in I$, we can write~$1^{\tau(i)}=(1+u)b(1+v)$ with~$(1+u),(1+v)\in S(\beta)$ and~$b\in B_{\beta}$. By Lemma~\ref{lemFindIdempotents} applied with~$\mf{k}_0=\mfa_0\cap B_\beta$ and~$\mf{k}_r=(S(\beta)-1)\cap\mfa_{r}$, there is an idempotent~$e$ in~$B_\beta$ which is congruent to~$b\pmod{S(\beta)-1}$. Since, in particular,~$e\equiv b\equiv 1^{\tau(i)}\equiv 1^i \pmod{\mfa_1}$, Lemma~\ref{lemCentralIdempoents} implies that~$e$ is a central idempotent in~$B_\beta$, in particular a sum of primitive central idempotents of~$B_\beta$. Since~$e\equiv 1^i\pmod{\mf{a}_1}$, we see that in fact~$e=1^i$. Thus in fact~$1^i\equiv 1^{\tau(i)}\pmod{S(\beta)-1}$ and we deduce that~$g=\sum_i 1^{\tau(i)}1^i$ is an element of~$S(\beta)$ with the required property.
\end{proof}

\begin{remark}
Recall that a semisimple character~$\theta$ is called \emph{simple} if there is a simple stratum~$[\Lambda,q,m,\beta]$ such that~$\theta\in \C(\Lambda,m,\beta)$; then, by Proposition~\ref{propBK3.5.1}, every semisimple stratum~$[\Lambda,q,m,\beta']$ such that~$\theta\in \C(\Lambda,m,\beta')$ has to be simple.  
\end{remark}

%
 

We would like to be able to get an analogue of strata induction for semisimple characters, for which we need 
\begin{itemize}
 \item the ``translation principle'' initially introduced for simple characters in~\cite[2.11]{bushnellKutzko:94}, and
 \item a result on ``derived characters'' (see Proposition~\ref{propDerivedCharacters} below).
\end{itemize}

\begin{theorem}\label{thmTranslationPrinciple}
Let~$[\Lambda,q,m+1,\gamma]$ and~$[\Lambda,q,m+1,\gamma']$ be semisimple strata with the same associated splitting such that 
\[
\C(\Lambda,m+1,\gamma)=\C(\Lambda,m+1,\gamma').
\]
Let~$[\Lambda,q,m,\beta]$ be a semisimple stratum with splitting~$V=\bigoplus_{i\in I}V^i$ such that~$[\Lambda,q,m+1,\beta]$ is equivalent to~$[\Lambda,q,m+1,\gamma]$ and~$\gamma$ is an element of~$\prod_{i\in I} A^{i,i}$. Then there exists a semisimple stratum~$[\Lambda,q,m,\beta']$ with splitting~$V=\bigoplus_{i'\in I'}V'^{i'}$ such that~$[\Lambda,q,m+1,\beta']$ is equivalent to~$[\Lambda,q,m+1,\gamma']$, with~$\gamma'\in \prod_{i'\in I'}A^{i',i'}$ and 
\[
\C(\Lambda,m,\beta)=\C(\Lambda,m,\beta').
\]
\end{theorem}

\begin{proof}
In~\cite[2.11]{bushnellKutzko:94} the statement and its proof are formulated for strict simple strata but, if we follow line by line the proof, we see that the result is valid for arbitrary simple strata. Moreover, the proof is also valid if~$[\Lambda,q,m+1,\gamma]$ is simple but we allow~$[\Lambda,q,m,\beta]$ to be a \emph{semisimple} stratum, provided we make one small change: we substitute the criterion of~\cite[5.5]{bushnellKutzko:94} for a fundamental stratum to be simple by the criterion of Proposition~\ref{propSemiSimplCritForFundStratum} for a fundamental stratum to be semisimple. Let us repeat the steps of the proof to point out the changes:
\begin{enumerate}
\item\label{proofTranslationPrinciple.i} As in~\cite[2.11]{bushnellKutzko:94}, we take two tame corestrictions~$s$ and~$s'$ for~$\gamma$ and~$\gamma'$ such that~$s(x)$ is congruent to~$s'(x)$ mod~$\mf{a}_{t+1}$ for all elements~$x$ of~$\mf{a}_t$ and all integers~$t$, see~\cite[5.2(iii)]{bushnellKutzko:94}.  
 \item We put~$c=\beta-\gamma$ and show, as in~\cite[5.3]{bushnellKutzko:94}, that~$[\Lambda,m+1,m,s(c)]$ is fundamental implies that~$[\Lambda,m+1,m,s'(c)]$ is fundamental too; however, we need that the latter stratum also satisfies the criterion on the maps~$m_{n,m+1,s'(c)}$ of Proposition~\ref{propSemiSimplCritForFundStratum}. Note that the same proposition implies that the maps~$m_{n,m+1,s(c)}$ do satisfy this criterion. 
 
 The tame corestrictions~$s$ and~$s'$ are surjective as maps from~$\mf{a}_t$ to~$\mf{b}_{\gamma,t}$
 and to~$\mf{b}_{\gamma',t}$, respectively, and thus we obtain an isomorphism of~$\kappa_F$-vector spaces~$\phi_t$ from~$\mf{b}_{\gamma,t}/\mf{b}_{\gamma,t+1}$ to~$\mf{b}_{\gamma',t}/\mf{b}_{\gamma',t+1}$, for all integers~$t$, by sending the class of~$s(x)$ to that of~$s'(x)$; note that this is well defined by~\ref{proofTranslationPrinciple.i}. Then~$m_{n,m+1,s'(c)}$ is equal to~$\phi_{n-m-1}\circ m_{n,m+1,s(c)}\circ\phi_n^{-1}$ and thus, varying~$n$, the maps~$m_{n,m+1,s'(c)}$ satisfy the additional criterion of Proposition~\ref{propSemiSimplCritForFundStratum}.
\item As in the part after (5.4) in~\cite[2.11]{bushnellKutzko:94}, we show that the algebras~$\mathcal{R}([\Lambda,m+1,m,s(c)])$ and~$\mathcal{R}([\Lambda,m+1,m,s'(c)])$ are isomorphic, which implies, by Proposition~\ref{propSemiSimplCritForFundStratum}, that~$[\Lambda,m+1,m,s'(c)]$ is equivalent to a semisimple stratum; therefore~$[\Lambda,q,m,\gamma'+c]$ is equivalent to a semisimple stratum~$[\Lambda,q,m,\beta']$ and we have 
\[
\C(\Lambda,m,\beta)=\psi_c\C(\Lambda,m,\gamma)=\psi_c\C(\Lambda,m,\gamma')=\C(\Lambda,m,\beta').
\]
\end{enumerate}

Let us now consider the general case. Since~$\gamma$ and~$\gamma'$ have the same associated splitting~$V=\bigoplus_{j\in J}V^j$, the previous case implies that we can find, for every~$j\in J$, an element~$\beta'^j$ satisfying the assertions, but for the~$j$th block. Now we put~$\beta'=\sum_{j\in J}\beta'^j$ to obtain 
\[
\C(\Lambda,m,\beta)=\C(\Lambda,m,\beta')
\]
directly from the definition of the set of semisimple characters, because 
\begin{itemize}
 \item the equality is true blockwise for the blocks for~$\gamma$, or, which is the same, the blocks for~$\gamma'$;
 \item~$H^{m+1}(\beta,\Lambda)=H^{m+1}(\beta',\Lambda)$ by~\cite[(3.5.9)]{bushnellKutzko:93};
 \item~$\C(\Lambda,m+1,\beta)=\C(\Lambda,m+1,\beta')$.
\end{itemize}
\end{proof}

\begin{proposition}\label{propDerivedCharacters}
Suppose~$m<q-1$ and let~$[\Lambda,q,m,\beta]$ and~$[\Lambda,q,m,\beta']$ be semisimple strata which have defining sequences with a common first element~$[\Lambda,q,m+1,\gamma]$. Suppose~$\theta\in \C(\Lambda,m,\beta)$ and~$\theta'\in \C(\Lambda,m,\beta')$ are semisimple characters which agree on restriction to~$H^{m+2}(\gamma,\Lambda)$, so that we can write~$\theta'=\theta_0\psi_{\beta'-\gamma}$ and~$\theta=\theta_0\psi_{\beta-\gamma+c}$, for some~$\theta_0\in \C(\Lambda,m,\gamma)$ and~$c\in \mf{a}_{-(m+1)}$. Let~$s_\gamma$ be a tame corestriction with respect to~$\gamma$.
\begin{enumerate} 
 \item For any~$g\in I(\theta,\theta')$ there are elements~$x,y\in S(\gamma)$ and~$g'\in B_\gamma$ such that~$g=xg'y$; moreover,~$g'$ intertwines~$\psi_{s_\gamma(\beta-\gamma+c)}$ with~$\psi_{s_\gamma(\beta'-\gamma)}$.
 \item For any~$g'\in I_{B_\gamma^\times}(\psi_{s_\gamma(\beta-\gamma+c)},\psi_{s_\gamma(\beta'-\gamma)})$, there are elements~$x,y$ of~$1+\mf{m}_{-k_0(\gamma,\Lambda)-m-1}$ such that~$xg'y$ intertwines~$\theta$ with~$\theta'$. \label{propDerivedCharacters.ii}
 \item\label{propDerivedCharacters.iii} If the characters~$\psi_{s(\beta-\gamma+c)}$ and~$\psi_{s(\beta'-\gamma)}$ are equal, then  there is~$z\in 1+\mf{m}_{-k_0(\gamma,\Lambda)-m-1}$ such that~$\theta^{z}=\theta'$.
\end{enumerate}
\end{proposition}

\begin{remark}\label{remPropDerivedCharacters}
The strategy of the proof of Proposition~\ref{propDerivedCharacters}\ref{propDerivedCharacters.ii}  is as follows: we  take~$x$ and~$y$ such that~$xg'y$ intertwines the stratum~$[\Lambda,q,m,\beta']$ with~$[\Lambda,q,m,\beta+c]$ (see also Proposition~\ref{propS3}) 
and prove that $xg'y$ intertwines $\theta$ with $\theta'$. Thus, if~$c$ is an element of~$\prod_i A^{i,i}$ and~$g'$ maps 
the splitting associated to~$\beta'$ to that of~$\beta$, then we can choose~$x\in \prod_i A^{i,i}$ and~$y\in \prod_{i'}A^{i'i'}$ which satisfy the assertions of Proposition~\ref{propDerivedCharacters}\ref{propDerivedCharacters.ii}. 
\end{remark}

\begin{proof}
We have~$H^{m+1}(\beta,\Lambda)=H^{m+1}(\beta',\Lambda)=H^{m+1}(\gamma,\Lambda)$ by~\cite[(3.1.9)]{bushnellKutzko:93}, so we just write~$H^{m+1}$.
\begin{enumerate}
 \item The decomposition~$g=xg'y$ follows directly from Proposition~\ref{propIntertwiningSemisimplCharacter}. We remark also that, by~\cite[(3.6.2) and~(3.1.15)(ii)]{bushnellKutzko:93}, the elements~$x$ and~$y$ normalize~$H^{m+1}$. Thus~$g'\in I(\theta^x,\theta'^{y^{-1}})$. By~\cite[(3.3.9)]{bushnellKutzko:93} we have 
\[
 \theta^x=\theta_0\psi_{x^{-1}\gamma x-\gamma}\psi_{\beta-\gamma+c},
\qquad\text{and}\qquad
  \theta'^{y^{-1}}=\theta_0\psi_{y\gamma y^{-1}-\gamma}\psi_{\beta'-\gamma}. 
\]
We have~$\psi_{x^{-1}\gamma x-\gamma}=\psi_{a_\gamma(x)}$ and~$\psi_{y\gamma y^{-1}-\gamma}=\psi_{-a_\gamma(y)}$ (as characters of~$H^{m+1}$) and thus their restrictions to~$\tilde{U}^{m+1}(\Lambda)\cap B_\gamma^\times$ are trivial. Thus, on~$\tilde{U}^{m+1}(\Lambda)\cap B_\gamma^\times$, we have 
\[
\theta^x=\theta=\theta_0\psi_{\beta-\gamma+c}=\theta_0\psi_{s_\gamma(\beta-\gamma+c)},
\]
and analogously for~$\theta^{y^{-1}}$. Since~$g'$ intertwines~$\theta^x$ with~$\theta^{y^{-1}}$ and~$\theta_0$ with itself, it also intertwines~$\psi_{s_\gamma(\beta-\gamma+c)}$ with~$\psi_{s_\gamma(\beta'-\gamma)}$.
\item If some element~$g'\in B_\gamma^\times$ intertwines~$\psi_{s_\gamma(\beta-\gamma+c)}$ with~$\psi_{s_\gamma(\beta'-\gamma)}$ then it intertwines the stratum~$[\Lambda,m+1,m,s_\gamma(\beta'-\gamma)]$ with~$[\Lambda,m+1,m,s_\gamma(\beta-\gamma+c)]$ and thus, by Proposition~\ref{propS3}, there are elements~$x,y$ of~$1+\mf{m}_{-k_0(\gamma,\Lambda)-m-1}$ such that~$g'$ intertwines~$[\Lambda,q,m,y\beta'y^{-1}]$ with~$[\Lambda,q,m,x^{-1}(\beta+c)x]$; that is,~$g'$ is an element of~$I(\psi_{x^{-1}(\beta+c)x},\psi_{y\beta'y^{-1}})$. Now we have 
\[
\psi_{x^{-1}(\beta+c)x}=\psi_{\beta+c-\gamma}\psi_{x^{-1}\gamma x-\gamma}\psi_\gamma,
\]
and an analogous equation for~$\psi_{y\beta'y^{-1}}$. Since~$g'$ intertwines each of~$\theta_0$ and~$\psi_\gamma|_{H^{m+1}}$ with themselves, we deduce that~$g'$ intertwines~$\theta^x$ with~$\theta'^{y^{-1}}$. 
\item This follows immediately from~(ii) applied to the identity element by putting~$z=xy$, which normalizes~$H^{m+1}$. 
\end{enumerate}
\end{proof}

\subsection{Semisimple characters for~$G$}

Suppose now that~$[\Lambda,q,m,\beta]$ is a skew-semisimple stratum and continue with the notation of the previous subsection. The adjoint anti-involution~$\sigma$ of the signed hermitian form~$h$ acts on~$\C(\Lambda,m,\beta)$ via 
\[
(\sigma\cdot\theta)(g):=\theta(\sigma(g^{-1})),\ g\in H^{m+1}(\beta,\Lambda).
\]

\begin{definition}
We define the set of semisimple characters~$\C_-(\Lambda,m,\beta)$ to be the set of all restrictions~$\theta|_{H^{m+1}(\Lambda,\beta)\cap G}$ where~$\theta$ run through all elements of~$\C(\Lambda,m,\beta)^{\sigma}$, the set of~$\sigma$-fixed points.

We call an element of~$\C_-(\Lambda,m,\beta)$ a \emph{semisimple character for~$G$}.
\end{definition}

\begin{remark}[{\cite[3.6]{stevens:05},~\cite[2.5]{stevens:01-2}}]\ 
\begin{enumerate}
 \item  The restriction map from~$\C(\Lambda,m,\beta)^{\sigma}$ to~$\C_-(\Lambda,m,\beta)$ is bijective, in particular injective.
 \item For two skew-semisimple strata~$[\Lambda,q,m,\beta]$ and~$[\Lambda,q,m,\beta']$, characters~$\theta\in \C(\Lambda,m,\beta)$ and~$\theta'\in \C(\Lambda,m,\beta')$, and~$g\in G$, the following conditions are equivalent:
 \begin{itemize}
  \item~$g\in I_G(\theta,\theta')$;
  \item~$g\in I_G(\theta|_{H^{m+1}(\Lambda,\beta)\cap G},\theta'|_{H^{m+1}(\Lambda,\beta)\cap G})$.
 \end{itemize}
\end{enumerate}
\end{remark}

We have an analogous description to that of Proposition~\ref{propIntertwiningSemisimplCharacter} of the intertwining of a semisimple character for~$G$.

\begin{proposition}[{\cite[3.27]{stevens:05}}]\label{propIntertwiningSemisimplCharacterOfG}
For~$\theta_-\in \C_-(\Lambda,m,\beta)$ a semisimple character of~$G$, we have
\[
I_G(\theta_-)=(S(\beta)\cap G)(B_\beta\cap G)(S(\beta)\cap G).
\]
\end{proposition}

For two skew-semisimple strata giving the same set of semisimple characters we have a stronger version of Proposition~\ref{propBK3.5.1}. 

\begin{corollary}
Let~$[\Lambda,q,m,\beta]$ and~$[\Lambda,q,m,\beta']$ be skew-semisimple strata such that the intersection~$\C_-(\Lambda,m,\beta)\cap \C_-(\Lambda,m,\beta')$ is non-empty, and let~$\tau: I\ra I'$ be the bijection given by Proposition~\ref{propBK3.5.1}, such that~$1^i\equiv 1^{\tau(i)}\pmod{\mf{a}_{1}}$. Then the spaces~$V^i$ and~$V'^{\tau(i)}$ are isomorphic as hermitian spaces, for all indices~$i\in I$. 
\end{corollary}

\begin{proof}
This follows directly from Proposition~\ref{propliftResidualIsometries} applied to the map~$v\mapsto \sum_{i\in I}1^{\tau(i)}1^iv$.
\end{proof}

We also get an analogue to~\cite[3.5.9]{bushnellKutzko:93} for semisimple characters for~$G$.

\begin{proposition}\label{propAnalogueBK359}
Suppose~$[\Lambda,q,m,\beta']$ and~$[\Lambda,q,m,\beta]$ are skew-semisimple strata with the same associated splitting, such that~$m>0$ and 
\[
\C(\Lambda,m,\beta)=\C(\Lambda,m,\beta').
\]
Then~$H^m(\beta)=H^m(\beta')$ and there is a skew-semisimple stratum~$[\Lambda,q,m,\beta'']$ equivalent to~$[\Lambda,q,m,\beta]$, with the same associated splitting, such that 
\[
\C(\Lambda,m-1,\beta'')=\C(\Lambda,m-1,\beta').
\]
\end{proposition}

\begin{proof}
The same proof as in the first part of~\cite[3.5.9]{bushnellKutzko:93} shows that~$H^{m}(\beta)=H^{m}(\beta')$. Now we take a character~$\theta$ in~$\C(\Lambda,m-1,\beta)^\sigma$ and a skew-symmetric element~$b$ of~$\mf{a}_{-m}(\Lambda)$ in~$\prod_i A^{i,i}$ such that~$\theta\psi_b$ is an element of~$\C(\Lambda,m-1,\beta')$. The same proof as in the second part of~\cite[3.5.9]{bushnellKutzko:93} shows that there is a semisimple stratum~$[\Lambda,n,m-1,\beta'']$ equivalent to~$[\Lambda,n,m-1,\beta+b]$ such that~$\beta''\in\prod_i A^{i,i}$. Since~$\beta+b$ is skew-symmetric~$\beta''$ can be chosen skew-symmetric, by~\cite[1.10]{stevens:01}. Then
\[
\C(\Lambda,m-1,\beta'')=\C(\Lambda,m-1,\beta)\psi_{b}
\]
has a non-trivial intersection with~$\C(\Lambda,m-1,\beta')$, and thus they equal by the analogue of~\cite[3.5.8]{bushnellKutzko:93}. 
\end{proof}

Next we obtain an analogue of the translation principle, Theorem~\ref{thmTranslationPrinciple}.

\begin{theorem}\label{thmTranslationPrincipleForG}
Let~$[\Lambda,q,m+1,\gamma]$ and~$[\Lambda,q,m+1,\gamma']$ be skew-semisimple strata with the same associated splitting such that 
\[
\C(\Lambda,m+1,\gamma)=\C(\Lambda,m+1,\gamma').
\]
Let~$[\Lambda,q,m,\beta]$ be a skew-semisimple stratum, with splitting~$V=\bigoplus_{i\in I}V^i$, such that~$[\Lambda,q,m+1,\beta]$ is equivalent to~$[\Lambda,q,m+1,\gamma]$ and~$\gamma$ is an element of~$\prod_{i\in I} A^{i,i}$. Then, there exists a skew-semisimple stratum~$[\Lambda,q,m,\beta']$, with splitting~$V=\bigoplus_{i'\in I'}V'^{i'}$, such that~$[\Lambda,q,m+1,\beta']$ is equivalent to~$[\Lambda,q,m+1,\gamma']$, with~$\gamma'\in \prod_{i'\in I'}A^{i'i'}$ and
\[
\C(\Lambda,m,\beta)=\C(\Lambda,m,\beta').
\]
\end{theorem}

\begin{proof}
The proof of~\cite[2.11]{bushnellKutzko:94} starts with an iterative application of~\cite[3.5.9]{bushnellKutzko:93}, which we can modify as in the proof of~Proposition~\ref{propAnalogueBK359} to obtain a skew analogue. But now the proof of Theorem~\ref{thmTranslationPrinciple} goes through on noting that all strata which occur are, or can be chosen to be, skew. 
\end{proof}

Finally, we get an analogue of Proposition~\ref{propDerivedCharacters}, with the same proof (replacing the reference to Proposition~\ref{propIntertwiningSemisimplCharacter} by Proposition~\ref{propIntertwiningSemisimplCharacterOfG}). 

\begin{proposition}\label{propDerivedCharactersForG}
Suppose~$m<q-1$ and let~$[\Lambda,q,m,\beta]$ and~$[\Lambda,q,m,\beta']$ be skew-semisimple strata which have defining sequences with a common first element~$[\Lambda,q,m+1,\gamma]$. Let~$\theta\in \C(\Lambda,m,\beta)^\sigma$ and~$\theta'\in \C(\Lambda,m,\beta')^\sigma$ be semisimple characters 
which agree on~$H^{m+2}(\Lambda,\gamma)$, so that we can write~$\theta'=\theta_0\psi_{\beta'-\gamma}$ and~$\theta=\theta_0\psi_{\beta-\gamma+c}$, for some~$\theta_0\in \C(\Lambda,m,\gamma)^\sigma$ and~$c\in \mf{a}_{-(m+1),-}$. Let~$s_\gamma$ be a~$\sigma$-equivariant tame corestriction with respect to~$\gamma$.
\begin{enumerate} 
 \item For any~$g\in I_G(\theta,\theta')$ there are elements~$x,y\in S(\gamma)\cap G$ and~$g'\in B_\gamma\cap G$ such that~$g=xg'y$; moreover,~$g'$ intertwines~$\psi_{s_\gamma(\beta-\gamma+c)}$ with~$\psi_{s_\gamma(\beta'-\gamma)}$.
 \item For any~$g'\in I_{B_\gamma\cap G}(\psi_{s_\gamma(\beta-\gamma+c)},\psi_{s_\gamma(\beta'-\gamma)})$, there are~$x,y\in (1+\mf{m}_{-k_0(\gamma,\Lambda)-m-1})\cap G$ such that~$xg'y$ intertwines~$\theta$ with~$\theta'$. 
 \item If~$\psi_{s_\gamma(\beta-\gamma+c)}=\psi_{s_\gamma(\beta'-\gamma)}$ then there is~$z\in (1+\mf{m}_{-k_0(\gamma,\Lambda)-m-1})\cap G$ such that~$\theta^{z}=\theta'$.
\end{enumerate}
\end{proposition}

%
%
%
%
%


\section{Matching and conjugacy for semisimple characters}

In this final section we prove that there is an analogue of the matching Proposition~\ref{propMatching} for semisimple characters which intertwine. One might think that this matching could just come from that for the underlying semisimple strata, but these do not necessarily intertwine so this is not possible. Then the sufficient condition~\eqref{eqConLambdaLambdaPrime} for an ``intertwining implies conjugacy'' result for semisimple strata is also sufficient for semisimple characters, also in the case of semisimple characters for~$G$. 

\subsection{For general linear groups}
For a semisimple character~$\theta\in \C(\Lambda,m,\beta)$, with decomposition~$V=\bigoplus_{i\in I}V^i$ associated to~$[\Lambda,q,0,\beta]$, we write~$\theta_i$ for the restriction of~$\theta$ to~$H^{m+1}(\Lambda^i,\beta_i)=H^{m+1}(\beta,\Lambda)\cap A^{i,i}$, for each index~$i\in I$.

\begin{theorem}\label{thmMatchingForChar}
Let~$\theta\in \C(\Lambda,m,\beta)$ and~$\theta'\in \C(\Lambda',m,\beta')$ be semisimple characters which intertwine. Then there is a unique bijection~$\zeta:I\ra I'$ such that there is an element~$g\in
\tilde{G}$ with
\begin{enumerate}
 \item\label{thmMatchingForChar.i} $g V^i=V'^{\zeta(i)}$, for all~$i\in I$; 
 \item\label{thmMatchingForChar.ii} $\theta_i^{g^{-1}}$ and~$\theta'_{\zeta(i)}$ intertwine, for all~$i\in I$. 
\end{enumerate}
Moreover, all elements~$g'\in\tilde{G}$ which satisfy the first property also satisfy the second. 
\end{theorem}

\begin{proof}
We remark first that the final sentence of the statement is clear. We continue with the uniqueness of~$\zeta$ under the assumption that the existence statement is proven. If there are two bijections from~$I$ to~$I'$ satisfying the assertions of the theorem then there are indices~$i_1,i_2\in I$ and~$i'\in I'$ such that~$\theta_{i_1}$ and~$\theta_{i_2}$ intertwine with~$\theta'_{i'}$. By~\ref{thmMatchingForChar.i}, we can conjugate~$\theta'_{i'}$ to~$V^{i_1}$ and to~$V^{i_2}$, and afterwards~$\theta'_{i'}\otimes \theta'_{i'}$ is the Levi-part (under an Iwahori decomposition) of a simple character, which intertwines with~$\theta_{i_1}\otimes\theta_{i_2}$. The index set of the latter two semisimple characters have different cardinalities and we obtain a contradiction. 

We now turn to the existence proof. First we reduce to the case of lattice chains, in fact to the case where both lattice chains are block-wise principal lattice chains -- that is, for each index~$i$ the dimension~$\dim_{\kappa_F}\Lambda^i_k/\Lambda^i_{k+1}$ is independent of~$k$. For that we repeat the~$\dagger$-construction~$\Lambda^{\dagger}:=\oplus_{j=0}^{e-1}(\Lambda-j)$, where~$e$ is the period of~$\Lambda$ (which we assume coincides with that or~$\Lambda'$), and~$\Lambda'^{\dagger}$ similarly. Let us remark that~$\Lambda^\dagger$ is the direct sum of the~$(\Lambda^i)^\dagger$. We will also need to use the notion of \emph{endo-equivalence} of simple characters, for which we refer the reader to~\cite{bushnellHenniart:96} and~\cite{broussousSecherreStevens:12}. 

By assumption,~$\theta$ and~$\theta'$ intertwine and thus~$\theta^\dagger$ and~$\theta'^\dagger$ intertwine. Assume that we have proven the existence of~$\zeta$ for the case of block-wise principal lattice chains. In particular we find an element~$g$ which maps, for each index~$i$, the vector space~$(V^{i})^\dagger$ to~$(V'^{\zeta(i)})^\dagger$, and then~$(\theta_i^\dagger)^{g^{-1}}$ and~$\theta_{\zeta(i)}^{' \dagger}$ intertwine. In particular, this implies that~$V^{\zeta(i)}$ and~$V^i$ have the same dimension and that~$\theta_i^\dagger$ and~$\theta_{\zeta(i)}^{'\dagger}$ are endo-equivalent. (More precisely, they are realizations of endo-equivalent ps-characters.) Thus there is an isomorphism~$g_i:V^i\to V'^{\zeta(i)}$ and, for any such, the simple characters~$\theta_i^{g_i^{-1}}$ and~$\theta'_{\zeta(i)}$ intertwine, since they are realizations of endo-equivalent ps-characters on the same space. Thus the element~$\sum_{i\in I} g_i$ has all the required properties. This finishes the proof of the reduction to the block-wise principal case. 

Now we assume we are in the block-wise principal case and prove the existence of~$\zeta$. We proceed via induction on~$m$, with the case~$m\geq \floor{\frac{q}{2}}$ following directly from Proposition~\ref{propMatching}. For~$m<\floor{\frac{q}{2}}$, let~$[\Lambda,q,m+1,\gamma]$ be a semisimple stratum equivalent to~$[\Lambda,q,m+1,\beta]$ with~$\gamma\in \prod_iA^{i,i}$, and similarly for~$[\Lambda',q,m+1,\gamma']$. We write~$J$ for the index set of the splitting of~$[\Lambda,q,m+1,\gamma]$, and similarly~$J'$. We have the character~$\theta_\gamma=\theta|_{H^{m+2}(\gamma,\Lambda)}\in\C(\Lambda,m+1,\gamma)$, and similarly~$\theta'_{\gamma'}$, and these characters intertwine. In particular, by induction, there are a bijection~$\zeta_\gamma:J\to J'$ and~$g\in\tilde{G}$ such that~$gV^j=V'^{\zeta_\gamma(j)}$ and~$\theta_{\gamma,j}^{g^{-1}}$ intertwines~$\theta'_{\gamma',\zeta_\gamma(j)}$ for all~$j\in J$, where~$\theta_{\gamma,j}=\theta_\gamma|_{H^{m+2}(\gamma_j,\Lambda^j)}$. Since~$g\Lambda^j$ and~$\Lambda'^{\zeta_\gamma(j)}$ are then principal lattice chains of the same period in the same space, they are conjugate so, changing~$g$, we may assume they are equal; that is,~$g\in\tilde U(\Lambda)$. 

In particular, conjugating everything by~$g$, we may assume that the strata~$[\Lambda,q,m+1,\gamma]$ and~$[\Lambda,q,m+1,\gamma']$ have the same splitting and~$\theta_{\gamma,j}$ intertwines~$\theta'_{\gamma',j}$. Moreover, since we can then prove the existence of~$\zeta$ separately for each block~$V^j$, we may assume that~$\theta_\gamma$ and~$\theta'_{\gamma'}$ are simple (and intertwine). Since~$\Lambda$ and~$\Lambda'$ are again principal lattice chains of the same period in the same space, they are conjugate, and we can assume~$\Lambda=\Lambda'$. Then~\cite[Theorem~3.5.11]{bushnellKutzko:93} implies that~$\theta_\gamma$ and~$\theta'_{\gamma'}$ are conjugate by an element of~$\tilde{U}(\Lambda)$ so, by conjugating, we can assume they are equal. By~\cite[3.5.8]{bushnellKutzko:93} we then have that
\[
H^{m+1}(\beta,\Lambda)=H^{m+1}(\gamma,\Lambda)=H^{m+1}(\gamma',\Lambda)=H^{m+1}(\beta',\Lambda).
\]
Thus we abbreviate~$H^{m+1}$, and similarly~$H^{m+2}$. By the translation principle Theorem~\ref{thmTranslationPrinciple}, we can find a semisimple stratum~$[\Lambda,q,m,\beta'']$ with splitting~$V=\bigoplus_{i\in I''}V''^{i''}$ such that 
\begin{itemize}
 \item~$[\Lambda,q,m+1,\beta'']$ equivalent to~$[\Lambda,q,m+1,\gamma']$;
 \item~$\C(\Lambda,m,\beta'')=\C(\Lambda,m,\beta)$ and 
 \item~$\gamma'\in\prod_{i''}A^{i''i''}$.
\end{itemize}
Now Proposition~\ref{propBK3.5.1} implies that we have a bijection~$\tau:I\to I''$ and~$h\in S(\beta)$ such that~$hV^i=V''^{i''}$. The element~$h$ normalizes~$\theta$, thus~$\theta^{h^{-1}}_i=\theta_{\tau(i)}$. In particular, we may replace~$\beta$ by~$\beta''$, since we can then compose the bijection~$\zeta:I''\to I$ that we find with~$\tau$ (and left multiply the~$g$ with~$h$); that is, we can assume that~$\gamma=\gamma'$.

Now let~$s_\gamma$ be a tame correstriction with respect to~$\gamma$. We write~$\theta$ and~$\theta'$ as in Proposition~\ref{propDerivedCharacters},
\[
 \theta=\theta_0\psi_{\beta-\gamma+c}, \qquad\text{and}\qquad
  \theta'=\theta_0\psi_{\beta'-\gamma},
\]
with~$\theta_0\in\C(\Lambda,m,\gamma)$ and~$c\in\mf{a}_{-(m+1)}$. Moreover, by Remark~\ref{remPropDerivedCharacters}, we can assume that~$c$ is decomposed by the splitting~$V=\bigoplus_{i\in I}V^i$. Since~$\theta_0$ and~$\psi_c\theta_0$ are both elements of~$\C(\Lambda,m,\gamma)$, both are intertwined by every element of~$B_\gamma^\times$; in particular, we deduce that the derived stratum~$[\Lambda,m+1,m,s_\gamma(c)]$ is intertwined by every element of~$B_\gamma^\times$ and thus~$s_\gamma(c)$ is an element of~$F[\gamma]+\mf{b}_{\gamma,-m}$, by~\cite[Lemma~2.4.11]{bushnellKutzko:93}. Then, since~$c,\beta$ are both decomposed by the splitting~$V=\bigoplus_{i\in I}V^i$, there is a semisimple stratum~$[\Lambda,m+1,m,\delta]$ equivalent to~$[\Lambda,m+1,m,s(\beta-\gamma+c)]$ with splitting~$V=\bigoplus_{i\in I}V^i$. Similarly, there is a semisimple stratum~$[\Lambda,m+1,m,\delta']$ equivalent to~$[\Lambda,m+1,m,s(\beta'-\gamma)]$ with splitting~$V=\bigoplus_{i\in I'}V^{i'}$.

By Proposition~\ref{propDerivedCharacters} there is an element of~$B_\gamma$ which intertwines~$[\Lambda,m+1,m,s(\beta'-\gamma)]$ with~$[\Lambda,m+1,m,s(\beta-\gamma+c)]$, so intertwines the semisimple strata~$[\Lambda,m+1,m,\delta]$ and~$[\Lambda,m+1,m,\delta']$. Then the matching for semisimple strata, Proposition~\ref{propMatching}, implies that there is~$g\in B_\gamma^\times$ which matches their splittings; indeed, since we are in the block-wise principal case, we may choose such~$g\in \tilde U(\Lambda)\cap B_\gamma$. In particular, conjugating by this element (which centralizes~$\gamma$), we may assume that~$I=I'$ and the strata~$[\Lambda,m+1,m,\delta]$ and~$[\Lambda,m+1,m,\delta']$ are intertwined by an element of~$B_\gamma\cap \prod_iA^{i,i}$. But then, by Proposition~\ref{propDerivedCharacters} again,~$\theta_i$ intertwines with~$\theta'_i$ for all~$i\in I$, which finishes the proof. 
\end{proof}

\begin{theorem}\label{thmintConjSemisimpleCharacters}
Let~$\theta\in \C(\Lambda,m,\beta)$ and~$\theta'\in \C(\Lambda,m,\beta')$ be semisimple characters which intertwine, let~$\zeta: I\ra I'$ be the matching given by Theorem~\ref{thmMatchingForChar}, and suppose that the condition~\eqref{eqConLambdaLambdaPrime} holds. Then~$\theta$ is conjugate to~$\theta'$~by an element of~$\tilde{U}(\Lambda)$.
\end{theorem}

\begin{proof}
We prove the theorem by induction on~$m$, with the case~$m=q-1$ following from Theorem~\ref{thmintConjSemisimple}. Thus assume~$m$ to be strictly smaller than~$q-1$ and then, by the induction hypothesis and conjugating, we may assume that~$H^{m+2}(\beta,\Lambda)=H^{m+2}(\beta',\Lambda)=:H^{m+2}$ and that the two characters agree on~$H^{m+2}$. We apply the translation principle Theorem~\ref{thmTranslationPrinciple}, to reduce to the situation where the strata have defining sequences with a common first element~$[\Lambda,n,m+1,\gamma]$. Now writing~$\theta,\theta'$ as in Proposition~\ref{propDerivedCharacters}, we get that the derived strata intertwine so, by Theorem~\ref{thmintConjSemisimple}, are conjugate by elements of~$\tilde U(\Lambda)\cap B_\gamma$. But then Proposition~\ref{propDerivedCharacters}\ref{propDerivedCharacters.iii} gives us an element of~$\tilde{U}(\Lambda)$ which conjugates~$\theta$ to~$\theta'$.
\end{proof}

\subsection{For classical groups}
If two characters~$\theta_-\in \C_-(\Lambda,m,\beta)$ and~$\theta'_-\in \C_-(\Lambda,m,\beta')$ intertwine then their lifts~$\theta\in \C(\Lambda,m,\beta)^\sigma$ and 
$\theta'\in \C(\Lambda,m,\beta')^\sigma$ intertwine and we get a matching~$\zeta:I\ra I'$ from Theorem~\ref{thmMatchingForChar}.
Let us state the main theorem:

\begin{theorem}\label{thmintConjSemisimpleCharactersForG}
Let~$\theta_-\in \C_-(\Lambda,m,\beta)$ and~$\theta'_-\in \C_-(\Lambda,m,\beta')$ be two semi\-simple characters of~$G$, which intertwine over~$G$, and assume that their matching satisfies~\eqref{eqConLambdaLambdaPrime}. Then,~$\theta_-$ and~$\theta'_-$ are~$\U(\Lambda)$-conjugate. 
\end{theorem}

\begin{proof}
The proof is completely the same as for Theorem~\ref{thmintConjSemisimpleCharacters}, using Theorems~\ref{thmintConjSkewSemisimple} and~\ref{thmTranslationPrincipleForG} and Proposition~\ref{propDerivedCharactersForG} in place of the equivalent statements for~$\tilde{G}$.
\end{proof}

We also conjecture a more natural version of the Matching Theorem~\ref{thmMatchingForChar} for~$G$. 

\begin{conjecture}
Let~$[\Lambda,q,m\beta]$ and~$[\Lambda',q,m\beta']$ be skew-semisimple strata and~$\theta\in \C(\Lambda,m,\beta)^\sigma$ and~$\theta'\in \C(\Lambda',m,\beta')^\sigma$ two semisimple characters which are intertwined by an element of~$G$. Let~$\zeta:I\ra I'$ be the matching from Theorem~\ref{thmMatchingForChar}. Then, there is an element~$g\in G$ such that~$gV^i=V'^{\zeta(i)}$, for all~$i\in I$. 
\end{conjecture}

\end{document}